\documentclass[preprint,12pt]{elsarticle}
\usepackage{amsfonts}
\usepackage{amsfonts,latexsym}
\usepackage{amsmath}
\usepackage{amsthm}
\usepackage{amssymb}
\usepackage{mathrsfs}
\usepackage{color}
\usepackage{lineno,hyperref}
\modulolinenumbers[5]

\journal{Journal of Differential Equations}

\newtheorem{theorem}{\indent Theorem}[section]
\newtheorem{proposition}[theorem]{\indent Proposition}
\newtheorem{definition}[theorem]{\indent Definition}
\newtheorem{lemma}[theorem]{\indent Lemma}
\newtheorem{assumption}{Assumption}[section]
\newtheorem{remark}[theorem]{\indent Remark}

\begin{document}

\begin{frontmatter}

\title{Smooth center-stable/unstable manifolds and foliations of stochastic evolution equations with non-dense domain}

\author[label1]{Zonghao Li}
\author[label2]{Caibin Zeng}
\author[label1]{Jianhua Huang\corref{cor1}}
\ead{jhhuang32@nudt.edu.cn}
\address[label1]{College of Science, National University of Defense Technology, Changsha 410073, China}
\address[label2]{School of Mathematics, South China University of Technology, Guangzhou 510640, China}
\cortext[cor1]{Corresponding author.}

\begin{abstract}
The current paper is devoted to  the asymptotic behavior of a class of stochastic PDE. More precisely, with the help of the theory of integrated semigroups and a crucial estimate of the random Stieltjes convolution, we study the existence and smoothness of center-unstable invariant manifolds and center-stable foliations for a class of stochastic PDE with non-dense domain through  the Lyapunov-Perron method. Finally, we give two examples about a stochastic age-structured model and a stochastic parabolic equation to illustrate our results.
\end{abstract}


\end{frontmatter}

%\linenumbers

\section{Introduction}\label{sec1}

In this paper, we shall study invariant manifolds and foliations of the following Stratonovich stochastic
evolution equation 
\begin{equation}\label{eq1}
	\left\{
	\begin{array}{l}
	du=Audt+F\left(u\right)dt+u\circ{dW},\\
	u\left(0\right)=u_0\in\overline{D\left(A\right)}. 	\\
	\end{array}\right.
\end{equation}
where $A:D\left(A\right)\subset{X}\to{X}$ is a linear operator and $\overline{D\left(A\right)}\neq{X}$, namely, its domain is non-densely defined. Indeed, $W$ is a one-dimensional Brownian motion and $F:\overline{D\left(A\right)}\to{X}$ is  globally Lipschitz continuous. Note that the Cauchy problem \eqref{eq1} is ill-posed if the Hille-Yosida theory for the $C_0$-semigroup generated by $A$ breaks down. Cauchy problems with a non-dense domain cover several types of differential equations, including delay differential equations, age-structured models and some  evolution equations with nonlinear boundary conditions. Since the classical variation of constants formula is not applicative, integrated semigroup theory is used to define a suitable integrated solution. The concept of integrated semigroups was firstly introduced by Arendt \cite{WA1987,WA2001} and it was applied to study the existence and uniqueness of solutions to such non-homogeneous Cauchy problems in a deterministic setting by Da Prato and Sinestrari \cite{DS1987}. Later on, Thieme \cite{HS1990,HS1990_} established the well-posedness of non-autonomous and semilinear Cauchy problems under a prior Hille-Yosida type estimate for the resolvent of non-densely defined operators. By discarding the mentioned estimate, Magal and Ruan \cite{PM2007,PM2009,PM_2009,Magal2016, MR2018} used the integrated semigroup theory to construct modified variation of constants formula for Cauchy problems with non-dense domain and applied the center manifold theroy to study Hopf bifurcation of age-structured models. Liu , Magal and Ruan studied center-unstable manifolds for Cauchy problems with non-dense domain and applications to stability of Hopf bifurcation\cite{LMR2012}. There are in general two methods to construct invariant manifolds, one is the graph transform method of Hadamard\cite{Ha1901}, and the other is Lypapunov-Perron method of Lyapunov\cite{Ly1947} and Perron\cite{Po1928}, which is an analytic method. Taking random factors into consideration, Neam\c{t}u \cite{AN2020} extended the mentioned theory of integrated semigroups to ill-posed stochastic evolution equations \eqref{eq1} driven by linear white noise and established the existence of random stable/unstable manifolds based on the Lyapunov-Perron method. Moreover, Li and Zeng \cite{LZ2021} studied center manifolds of stochastic evolution equations with non-dense domain driven by linear noise. Wang\cite{BA2021}, Li \emph{et al.}\cite{LZms} studied mean-square invariant manifolds based on mean-square random dynamical systems, presented by Lorenz and Kloeden \cite{LK2012}. Therefore, the mentioned work brings much interest to discern the long-time dynamics of \eqref{eq1}.

There are several results regarding invariant manifolds in the framework of random dynamical systems(RDSs). Mohammed and Scheutzow \cite{Mh1999} studied the existence of local stable and unstable manifolds of stochastic differential equations driven by semimartingales. Duan \emph{et al.} \cite{Duan2004}, Lu and Schmalfu{\ss} \cite{KB2007} and Caraballo \emph{et al.} \cite{Cara2010} studied stable and unstable manifolds for stochastic partial differential equations. Lian and Lu \cite{LL2010} proved a multiplicative ergodic theorem and then use this theorem to establish the stable and unstable manifold theorem for nonuniformly hyperbolic random invariant sets. Li \emph{et al.} \cite{PJ2013} proved the persistence of smooth normally hyperbolic invariant manifolds for dynamical systems under random perturbations. Chen \emph{et al.} \cite{Chen2015} studied center manifolds for stochastic partial differential equations under an assumption of exponential trichotomy. Shi \cite{Shi2020} studied the limiting behavior of center manifolds for a class of singularly perturbed stochastic partial differential equations in terms of the phase spaces. 

The research about invariant foliations for deterministic dynamical systems dates back to the work by Finichel\cite{FN1971} and Pesin\cite{YP1977}. In the framwork of RDSs,  Liu and Qian\cite{LQ1995} studied stable and unstable foliations for finite-dimensional RDS, Lu and Schmalfu{\ss}\cite{KB2008} studied the local theory of invariant foliations for SPDEs. Zeng and Shen \cite{SZ2021} obtained the existence of the invariant stable and unstable foliations of \eqref{eq1}.  For a normally hyperbolic random invariant
manifold, Li \emph{et al.}\cite{PJ2014} proved the existence of invariant foliations of the associated stable and unstable manifolds.  By using the Wong-Zakai approximation, Shen \emph{et al.}\cite{SJKB} proved that the invariant manifolds and stable foliations of the Wong-Zakai approximations converge to those of the Stratonovich stochastic evolution equation, respectively.

With the help of the theory of integrated semigroups, we construct a new variation of constants formula by the resolvent operator  connecting $\overline{D\left(A\right)}$ and $X$. Unlike the argument using the classical Lyapunov-Perron method, the spectral gap condition contains a crucial constant in the derivation of estimation of random Stieltjes-type convolution. The required conditions for smoothness do not appear
to be easily verifiable at first sight.We present a positive answer to the smoothness problem for invariant manifolds and foliations of stochastic evolution
equations \eqref{eq1}. To be precise, we shall establish the conditions for $C^k$-smoothness random center-unstable invariant manifolds and center-stable foliations of \eqref{eq1}.

The remaining part of this article is structured as follows. In  Section \ref{sec2}, we collect some essential results on random dynamical systems and integrated semigroups, and layout the basic assumptions. Then in Section \ref{sec3} and Section \ref{sec4}, we prove the existence and smoothness of invariant center-unstable manifolds and invariant center-stable foliations of \eqref{eq1} respectively. Finally in Section \ref{sec5}, we give two illustrative examples about a age-structured model and a stochastic parabolic equation to verify the main results in the present paper.

\section{Preliminaries}\label{sec2}
\subsection{Random Dynamical Systems}
Consider flows on the probability space $\left(\Omega, \mathcal{F}, \mathbb{P}\right)$, where a flow $\theta : \mathbb{R}\times\Omega\to\Omega$ of $\{{\theta_t}\}_{t\in\mathbb{R}}$ is defined on the sample space $\Omega$ having following properties: 

\begin{itemize}
	\item [$\left(\mathrm{i}\right)$] the mapping $\left(t,\omega\right)\mapsto\theta_{t}\omega$ is $\left(\mathcal{B}\left(\mathbb{R}\right)\otimes\mathcal{F},\mathcal{F}\right)$-measurable.
	 
	\item [$\left(\mathrm{ii}\right)$] $\theta_0=I_{\Omega}$.
 
	\item [$\left(\mathrm{iii}\right)$] $\theta_{t+s}=\theta_t\circ\theta_s$ for all $t,s\in\mathbb{R}$.
\end{itemize}

In addition, we assume that $\mathbb{P}$ is ergodic with respect to $\{{\theta_t}\}_{t\in\mathbb{R}}$. Then the quadruple $\left(\Omega, \mathcal{F}, \mathbb{P},\{{\theta_t}\}_{t\in\mathbb{R}}\right)$ is called a metric dynamical system.  Usually, let $W\left(t\right)$ be a two-sided Wiener process with trajectories in the space $\Omega:=C_0\left(\mathbb{R},\mathbb{R}\right)$ of real continuous functions defined on $\mathbb{R}$ which is a set equipped with compact open topology $\mathcal{F}:=\mathcal{B}\left(C_0\left(\mathbb{R},\mathbb{R}\right)\right)$. On this set we consider the measurable flow $\{{\theta_t}\}_{t\in\mathbb{R}}$ defined by $\theta_t\omega\left(\cdot\right)=\omega\left(\cdot+t\right)-\omega\left(t\right)$ which is called the Wiener shift, the associated distribution $\mathbb{P}$ is a Wiener measure defined on $\mathcal{F}$. A set $\Omega$ is called $\{{\theta_t}\}_{t\in\mathbb{R}}$-invariant if $\theta_t\Omega=\Omega$ for $t\in\mathbb{R}$. For more details, see \cite{Arnold1998}. Using these notations, we could present the definition of a random dynamical system.
\begin{definition}\label{def2.1}
	Let $X$ be a separable Hilbert space. A continuous random dynamical system on  $X$ over a metric dynamical system $\left({\Omega}, \mathcal{F}, \mathbb{P},\{{\theta_t}\}_{t\in\mathbb{R}}\right)$ is a mapping 
	\begin{equation*}
		\varphi:\mathbb{R}\times\Omega\times{X}\to{X},  ~\left(t,\omega,x\right)\mapsto\varphi\left(t,\omega,x\right),
		\end{equation*}
which is $\left(\mathcal{B}\left(\mathbb{R}\right)\otimes\mathcal{F}\otimes\mathcal{B}\left(X\right),\mathcal{B}\left(X\right)\right)$-measurable and satisfies:
	\begin{itemize}
		\item [$\left(\mathrm{i}\right)$] $\varphi\left(0,\omega,x\right)=x$, for $x\in{X}$ and $\omega\in{\Omega}$.
	 
		\item  [$\left(\mathrm{ii}\right)$] $\varphi\left(t+s,\omega,x\right)=\varphi\left(t,\theta_s\omega,\varphi\left(s,\omega,x\right)\right)$, for all $t,s\in\mathbb{R},x\in X$ and all $\omega\in {\Omega}$.
	\end{itemize}
\end{definition}
\begin{definition}\label{def2.2}
	A random set $\mathcal{M}\left(\omega\right)$ is called invariant for a random dynamical system $\varphi\left(t,\omega,x\right)$ if $\varphi\left(t,\omega,\mathcal{M}\left(\omega\right)\right)\subset{\mathcal{M}\left(\theta_t\omega\right)}$ for $t\geq0$.
\end{definition}
We will use a coordinate transform to convert the stochastic partial differential equation \eqref{eq1} into a random evolution equation \cite{Duan2004,Cara2013,Duan2003}. To this end, we consider the linear stochastic differential equation 
\begin{equation}\label{eq2.1}
	dz+{\mu}zdt=dW, 
\end{equation}
where $\mu$ is a positive parameter. We extract the important results on the stationary Ornstein-Uhlenbeck process in the following. 
\begin{lemma}\cite[Lemma 2.1]{Duan2003}\label{lemma2.1}
	\begin{itemize}
		\item[$\left(\mathrm{i}\right)$]There exists a $\{{\theta_t}\}_{t\in\mathbb{R}}$-invariant set $\tilde{\Omega}\subset{C_0\left(\mathbb{R},\mathbb{R}\right)}$ of full measure satisfying the sublinear growth condition
		\begin{equation*}
			\lim_{t\to\pm\infty}\frac{\left|\omega\left(t\right)\right|}{\left|t\right|}=0,~ \omega\in \tilde{\Omega}.
		\end{equation*}
		\item [$\left(\mathrm{ii}\right)$]The random variable $z\left(\omega\right)=-\int_{-\infty}^{0}e^{{s}}\omega\left(s\right)ds$ is well defined for all $\omega\in\tilde{\Omega}$ and generates a unique stationary solution of \eqref{eq2.1} given by 
		\begin{equation*}
			\left(t,\omega\right)\mapsto{z\left(\theta_t\omega\right)}=-\int_{-\infty}^{0}e^{{ s}}\theta_t\omega\left(s\right)ds=-\int_{-\infty}^{0}e^{{s}}\omega\left(s+t\right)ds+\omega\left(t\right).
		\end{equation*}
		The mapping $t\mapsto{z\left(\theta_t\omega\right)}$ is continuous.
		\item [$\left(\mathrm{iii}\right)$]Moreover
	\begin{equation*}
		\lim_{t\to\pm\infty}\frac{\left|z\left(\theta_t\omega\right)\right|}{\left|t\right|}=0,
	\end{equation*}
	for all $\omega\in \tilde{\Omega}$.
 
		\item [$\left(\mathrm{iv}\right)$]In addition 
\begin{equation*}
	\lim_{t\to\pm\infty}\frac{1}{t}\int_{0}^{t}z\left(\theta_s\omega\right)ds=0,
\end{equation*}
		for all $\omega\in \tilde{\Omega}$.		
	\end{itemize}
\end{lemma}
\begin{proof}
	See \cite[Lemma 2.2]{PZL2023}.
\end{proof}
From now on, we will work on $\left(\tilde{\Omega}, \tilde{\mathcal{F}}, \tilde{\mathbb{P}},\left\{{\theta_t}\right\}_{t\in\mathbb{R}}\right)$, where $\tilde{\mathcal{F}}$ is the trace algebra of $\tilde{\Omega}$ and the restriction of Wiener measure to this $\sigma$-algebra $\tilde{\mathcal{F}}$ is denoted by $\tilde{\mathbb{P}}$. With a slight abuse of notation, we again denote the metric dynamical system by $\left({\Omega}, {\mathcal{F}}, \mathbb{P},\left\{{\theta_t}\right\}_{t\in\mathbb{R}}\right)$ in the following context.

\subsection{Integrated Semigroups}
In this subsection, we collect some basic results about non-densely defined operators and integrated semigroups from \cite{PM2007,PM2009,PM_2009}. Since $A$ is a non-densely defined linear operator in a Hilbert space $X$, its resolvent set is denoted by $\rho\left(A\right)=\left\{\lambda\in\mathbb{C}:\lambda{I}-A \text{~is invertible}\right\}$, and $\rho\left(A\right)\neq\emptyset$. The spectrum of $A$ is denoted by $\sigma\left(A\right):=\mathbb{C}\backslash\rho\left(A\right)$. Moreover, let $X_0=\overline{D\left(A\right)}$ equipped with the norm $\Vert\cdot\Vert$ and construct the part of $A$ in $X_0$ denoted by $A_0:D\left(A_0\right)\subset{X_0}\to{X_0}$, which is a linear operator on $X_0$ defined by 
\begin{equation*}
A_0x=Ax,\forall{x\in{D\left(A_0\right)}:=\left\{y\in{D\left(A\right):Ay\in{X_0}}\right\}}.	
\end{equation*}
Assume that there exists a constant $\vartheta$ satisfying $\left(\vartheta,+\infty\right)\subset\rho\left(A\right)$, it is easy to check that for each $\lambda>\vartheta$, 
\begin{equation*}
	D\left(A_0\right)=\left(\lambda{I}-A\right)^{-1}X_0,  ~{{\left(\lambda{I}-A_0\right)}^{-1}=\left(\lambda{I}-A\right)^{-1}|}_{X_0}.
\end{equation*}
Therefore, it follows from \cite[Lemma 2.1]{PM_2009} that $\rho\left(A\right)=\rho\left(A_0\right)$, and thus $\sigma\left(A\right)=\sigma\left(A_0\right)$. 

We next discuss some basic assumptions about operator $A$ and $A_0$. Assume that $A_0$ is a Hille-Yosida operator in $X_0$, i.e., there exist two constants, $\vartheta\in\mathbb{R}$ and $M\geq1$, such that $\left(\vartheta,+\infty\right)\subset\rho\left(A_0\right)$ and
\begin{equation*}
	\left\Vert{\left(\lambda{I}-A_0\right)}^{-k}\right\Vert_{{\mathcal L}\left(X_0\right)}\leq\frac{M}{{\left(\lambda-\vartheta\right)}^{k}},~\forall\lambda>\vartheta,~\forall{k}\geq1.
\end{equation*}
And we assume
\begin{equation*}
	\lim_{\lambda\to+\infty}{\left(\lambda{I}-A\right)}^{-1}x=0,\forall{x\in{X}}.
\end{equation*}
\begin{lemma}\label{th3.1}\cite[Theorem 5.3]{APazy1983} A linear operator $A_0:D\left(A_0\right)\subset{X_0}\to{X_0}$ is an infinitesimal generator of a $C_0$-semigroup ${\left(T\left(t\right)\right)}_{t\geq0}$ satisfying $\left\Vert{T\left(t\right)}\right\Vert\leq{Me^{\vartheta{t}}}$, if and only if
\begin{itemize}
	\item [$\left(\mathrm{i}\right)$] $A_0$ is densely defined in $X_0$;
	\item [$\left(\mathrm{ii}\right)$] $A_0$ is a Hille-Yosida operator in $X_0$.
\end{itemize}
\end{lemma}

\begin{lemma}\label{lemma3.1}
\cite[Lemma 2.1]{PM2007} Let $\left(X,\left\Vert\cdot\right\Vert\right)$ be a Banach space and $A:D\left(A\right)\subset{X}\to{X}$ be a linear operator. Assume that there exists $\vartheta\in{\mathbb{R}}$, such that $\left(\vartheta,+\infty\right)\subset\rho\left(A\right)$ and 
\begin{equation*}
	\limsup_{\lambda\to+\infty}\lambda{\left\Vert{\left(\lambda{I}-A\right)}^{-1}\right\Vert}_{{\mathcal L}\left(X_0\right)}<+\infty.
\end{equation*}
Then the following assertions are equivalent:
\begin{itemize}
	\item [$\left(\mathrm{i}\right)$] $\lim_{\lambda\to+\infty}\lambda{\left(\lambda{I}-A\right)}^{-1}x=x,\forall{x\in{X_0}}$.
 
	\item [$\left(\mathrm{ii}\right)$] $\lim_{\lambda\to+\infty}{\left(\lambda{I}-A\right)}^{-1}x=0,\forall{x\in{X}}$.
 
	\item [$\left(\mathrm{iii}\right)$] $\overline{D\left(A_0\right)}=X_0$.
\end{itemize}
\end{lemma}
 Now we are ready to introduce the definition of integrated semigroups.
\begin{definition}\label{def2.4}
	Let $X$ be a Banach space. A family of bounded linear operators $\{S\left(t\right)\}_{t\geq0}$ on $X$ is called an integrated semigroup if
	\begin{itemize}
		\item[$\left(\mathrm{i}\right)$]$S\left(0\right)=0$.
	 
		\item[$\left(\mathrm{ii}\right)$]The map $t\to{S\left(t\right)}x$ is continuous on $[0,+\infty)$ for each $x\in{X}$.
 
		\item[$\left(\mathrm{iii}\right)$]$S\left(t\right)$ satisfies
		\begin{equation*}
			S\left(s\right)S\left(t\right)=\int_{0}^{s}\left(S\left(r+t\right)-S\left(r\right)\right)dr, \forall{s,t\geq0}.
		\end{equation*}
	\end{itemize}	
\end{definition}

According to Lemma \ref{th3.1} and Lemma \ref{lemma3.1}, $A_0$ generates a $C_0$-semigroup on $X_0$ denoted by ${\left(T\left(t\right)\right)}_{t\geq0}$. It can also be denoted by $\left(e^{A_0t}\right)_{t\geq0}$. Besides, as concluded in \cite[Proposition 2.5]{PM2009}, $A$ generates an integrated semigroup on $X$ denoted by ${\left(S\left(t\right)\right)}_{t\geq0}$ and for each $x\in{X}, t\geq0$, ${\left(S\left(t\right)\right)}_{t\geq0}$ is given by
\begin{equation}\label{eq2.2}
	S\left(t\right)x=\lambda\int_{0}^{t}T\left(s\right){\left(\lambda{I}-A\right)}^{-1}xds+\left(I-T\left(t\right)\right){\left(\lambda{I}-A\right)}^{-1}x.
\end{equation}
Also, the map $t\to{S\left(t\right)}x$ is continuously differentiable if and only if $x\in{X_0}$ and 
\begin{equation*}
	\frac{dS\left(t\right)x}{dt}=T\left(t\right)x, ~\forall{t\geq0},\ \forall{x\in{X_0}}.
\end{equation*}

Next we introduce the following definition of MR operator which is named by Magal and Ruan\cite{PM2007}.
\begin{definition}[MR operator]
	Let $A$ be a closed linear operator in a Banach space $X$ with $\overline{D\left(A\right)}\neq X$ and $A_0$ be the part of $A$ in $X_0=\overline{D\left(A\right)}$. $A$ is called a MR operator if 
	\begin{itemize}
		\item [$\left(\mathrm{i}\right)$] $A_0$ generates a $C_0$-semigroup ${\left(T\left(t\right)\right)}_{t\geq0}$ on $X_0$, i.e. $A_0$ is a Hille-Yosida operator in $X_0$, and $\overline{D(A_0)}=X_0$. Then $A$ generates an integrated semigroup ${\left(S\left(t\right)\right)}_{t\geq0}$ in $X$.
		\item [$\left(\mathrm{ii}\right)$]Assume that there exists real number $\tau_0>0$, and a non-decreasing map $\delta:[0,\tau_0]\to[0,+\infty)$ such that for each $f\in{C^1\left([0,\tau_0];X\right)}$,
\begin{equation*}
	\left\Vert\frac{d}{dt}\int_0^t S\left(t-s\right)f\left(s\right)ds\right\Vert\leq\delta\left(t\right)\underset{s\in[0,t]}{\sup}\left\Vert{f\left(s\right)}\right\Vert,~\forall{t\in[0,\tau_0]},
\end{equation*}	
where $\delta$ satisfies
\begin{equation*}
	\lim_{t\to0^+}\delta\left(t\right)=0.
\end{equation*}

	\end{itemize}
\end{definition}
We assume that $A$ is a MR operator from now on.
\begin{lemma}\cite[Lemma 2.6]{PM2007}\label{lemma3.2}
Let $\tau_0>0$ be fixed. 	For each $f\in{C^1\left([0,\tau_0];X\right)}$, define
\begin{equation*}
	\left(S*f\right)\left(t\right)=\int_0^t{S\left(t-s\right)f\left(s\right)}ds,~\forall{t}\in[0,\tau_0].
\end{equation*}
Then we have the following:
\begin{itemize}
	\item[$\left(\mathrm{i}\right)$]The map $t\to{\left(S*f\right)\left(t\right)}$ is continuously differentiable on $[0,\tau_0]$.
	
	\item[$\left(\mathrm{ii}\right)$]$\left(S*f\right)\left(t\right)\in{D\left(A\right)},~\forall{t\in[0,\tau_0]}$.
	\item[$\left(\mathrm{iii}\right)$]if we set $u\left(t\right)=\frac{d}{dt}\left(S*f\right)\left(t\right)$, then
		\begin{equation*}
			u\left(t\right)=A\int_0^t{u\left(s\right)}ds+\int_0^t{f\left(s\right)}ds, ~\forall{t\in\left[0,\tau_0\right]}.
		\end{equation*}
	\item[$\left(\mathrm{iii}\right)$]For each $\lambda>\vartheta$, and each $t\in[0,\tau_0]$, we have
\begin{equation*}
	{\left(\lambda{I}-A\right)}^{-1}\frac{d}{dt}\left(S*f\right)\left(t\right)=\int_0^tT\left(t-s\right){\left(\lambda{I}-A\right)}^{-1}f\left(s\right)ds.
\end{equation*}
\end{itemize}
\end{lemma}

 For each $f\in{C^1\left([0,\tau_0];X\right)}$, we denote
\begin{equation*}
	\left(S\diamond{f}\right)\left(t\right)=\frac{d}{dt}\left(S*f\right)\left(t\right),~\forall{t\in[0,\tau_0]},
\end{equation*}
which is defined as a Stieltjes convolution.
According to Lemma \ref{lemma3.2} and Lemma \ref{lemma3.1}$\left(i\right)$, we have 
\begin{equation*}
	\left(S\diamond{f}\right)\left(t\right)=\lim_{\lambda\to+\infty}\int_{0}^{t}T\left(t-s\right)\lambda{\left(\lambda{I}-A\right)}^{-1}f\left(s\right)ds,~\forall{t\in[0,\tau_0]},~\forall\lambda>\vartheta.
\end{equation*} 

\begin{lemma}\label{lem2.8}\cite[Proposition 2.13]{PM_2009}
For $\kappa>\vartheta$, there exists $C_\kappa>0$ such that $f\in{C\left(\mathbb{R}^+;X\right)}$ and
\begin{equation*}
	\left\Vert\left(S\diamond{f}\right)\left(t\right)\right\Vert\leq{C_\kappa}\underset{s\in[0,t]}{\sup}e^{\kappa\left(t-s\right)}\left\Vert{f\left(s\right)}\right\Vert,~\forall{t\geq0}.
\end{equation*} 
Moreover, for each $\varepsilon>0$, if $\tau_\varepsilon>0$ satisfying $M\delta\left(\tau_\varepsilon\right)\leq\varepsilon$, it holds
\begin{equation*}
	C_\kappa=\frac{2\varepsilon{\max\left(1,e^{-\kappa\tau_\varepsilon}\right)}}{1-e^{\left(\vartheta-\kappa\right)\tau_\varepsilon}}. 
\end{equation*}
\end{lemma}

\begin{assumption}\label{as2.1}
$F: X_0 \to X$ is globally Lipschitz continuous, i.e.
\begin{equation*}
	\left\Vert F\left(u_1\right)-F\left(u_2\right)\right\Vert\leq{L\left\Vert u_1-u_2\right\Vert},
\end{equation*} 
 for any $u_1,u_2\in X_0$, where $L$ is the Lipschitz constant and $F\left(0\right) = 0$.
\end{assumption}

To our aim, we define the coordinate transform 
\begin{equation*}
	v=\Xi\left(u,\omega\right)=ue^{-z\left(\theta_t\omega\right)}, 
\end{equation*}   
and its reverse transform 
\begin{equation*}
	u=\Xi^{-1}\left(v,\omega\right)=ve^{z\left(\theta_t\omega\right)},
\end{equation*}
for $t\geq0$. Performing the transformation $\Xi\left(v,\omega\right)$ into \eqref{eq1}, we obtain
\begin{equation}\label{eq4}
	\left\{
	\begin{array}{l}
	dv=Avdt+z\left(\theta_t\omega\right)vdt+G\left({\theta_t}\omega,v\right)dt, \\
	v\left(0\right)=x\in{X_0},  	\\
	\end{array}\right.
\end{equation}
where $G\left(\omega,v\right)=e^{-z\left(\omega\right)}F\left(e^{z\left(\omega\right)}v\right)$. Obviously,  $G$ and $F$ have the same Lipschitz constant.

\begin{definition}\label{def2.9}
	A mapping $v\in{C\left(\mathbb{R}^+;X\right)}$ is called an integrated solution of \eqref{eq4} if and only if for $\forall{t\geq0}$,
	\begin{itemize}
		\item [$\left(\mathrm{i}\right)$]$\int_0^tv\left(s\right)ds\in{D\left(A\right)}$.
		
		\item [$\left(\mathrm{ii}\right)$] It holds
\begin{equation*}
	v\left(t\right)=x+A\int_0^tv\left(s\right)ds+\int_0^tz\left(\theta_s\omega\right)v\left(s\right)ds+\int_0^tG\left({\theta_s}\omega,v\left(s\right)\right)ds.
\end{equation*}
	\end{itemize}
\end{definition}

For the linear part of \eqref{eq4}, we define
\begin{equation*}
	\phi_{A_0}\left(t,s\right)=e^{A_0\left(t-s\right)+\int_{s}^{t}z\left(\theta_r\omega\right)dr}=T\left(t-s\right)e^{\int_{s}^{t}z\left(\theta_r\omega\right)dr}.
\end{equation*}

\subsection{Exponential Trichotomy}
%Recall that $A_0:D\left(A_0\right)\subset{X_0}\to{X_0}$ generates a $C_0$-semigroup $\{T\left(t\right)\}_{t\geq0}$ on $X_0$. 

\begin{assumption}\label{as2.2}
	We assume that $T\left(t\right)$ satisfies the exponential trichotomy with exponents $\alpha>\gamma\ge0\ge-\gamma>-\beta$ and bound $K\geq1$, more precisely,
	\begin{itemize}
		\item [$\left(\mathrm{i}\right)$]there exist two continuous linear projection operators with finite-rank $\Pi_{0c}\in{L}\left(X_0\right)\backslash\{0\}$ and $\Pi_{0u}\in{L}\left(X_0\right)$ such that 
		\begin{equation*}
\Pi_{0u}\Pi_{0c}=\Pi_{0c}\Pi_{0u}=0, ~\Pi_{0k}T\left(t\right)=T\left(t\right)\Pi_{0k},~\forall{t\geq0},~\forall{k\in\{c,u\}}.
\end{equation*}
Denote 
\begin{equation*}
	\Pi_{0s}=I_{X_0}-\left(\Pi_{0c}+\Pi_{0u}\right),~X_{0k}=\Pi_{0k}X_0,~k\in\{c,u,s\}.
	\end{equation*}
	\item [$\left(\mathrm{ii}\right)$]Denote $X_{0ij}=X_{0i}\oplus X_{0j}, \Pi_{0ij}=\Pi_{0i}+\Pi_{0j}~i,j\in\{c,u,s\},~i\neq j$. Then $\Pi_{0cu}T\left(t\right)=T\left(t\right)\Pi_{0cu}$. The restriction $T\left(t\right)|_{R\left(\Pi_{0cu}\right)},~t\geq0$ is an isomorphism of the range $R\left(\Pi_{0cu}\right)$ of $\Pi_{0cu}$ onto itself and define $T\left(t\right),~t<0$ as the inverse map.
\item [$\left(\mathrm{iii}\right)$] For all $x\in{X_0}$, it holds 
\begin{equation}\label{eq5}
{\left\Vert{T}\left(t\right)\Pi_{0c}x\right\Vert}\leq{K}e^{\gamma|t|}\left\Vert{	\Pi_{0c}x}\right\Vert,~\forall{t\in\mathbb{R}},
\end{equation}
\begin{equation}\label{eq6}
{\left\Vert{T}\left(t\right)\Pi_{0u}x\right\Vert}\leq{K}e^{\alpha{t}}\left\Vert{\Pi_{0u}x}\right\Vert,~\forall{t\leq0},
\end{equation}
\begin{equation}\label{eq7}
{\left\Vert{T}\left(t\right)\Pi_{0s}x\right\Vert}\leq{K}e^{-\beta{t}}\left\Vert{\Pi_{0s}x}\right\Vert,~\forall{t\geq0}, 
\end{equation}
\end{itemize}
\end{assumption}
\begin{remark}\label{rem}
Assume $\alpha>\gamma\ge0\ge-\gamma>-\beta$. If $X_0$ is a finite dimensional space and there exist eigenvalues with real part being zero, less than zero and greater than zero, the exponential trichotomy holds. If $X_0$ is an infinite dimensional space and the spectrum of $A_0$ satisfies $\sigma\left(A_0\right)=\sigma^{0s}\cup\sigma^{0c}\cup\sigma^{0u}$, where $\sigma^{0s}={\left\{\lambda\in\sigma\left(A_0\right):\rm Re\left(\lambda\right)\leq-\beta\right\}}$, $\sigma^{0c}={\left\{\lambda\in\sigma\left(A_0\right):\left|\rm Re\left(\lambda\right)\right|\leq\gamma\right\}}$, $\sigma^{0u}={\left\{\lambda\in\sigma\left(A_0\right):\rm Re\left(\lambda\right)\geq\alpha\right\}}$, and $A_0$ generates a strong continuous semigroup, then the exponential trichotomy holds too (see \cite[page 267]{TG1993}. And if the spectrum of $A_0$ satisfies $\sigma\left(A_0\right)=\sigma^{0s}\cup\sigma^{0u}$, where $\sigma^{0s}$ and $\sigma^{0u}$ are the same as above, then the exponential dichotomy holds. Since $\sigma\left(A\right)=\sigma\left(A_0\right)$, the spectrum of $A$ can be split as $A_0$. 
\end{remark}
Since we only require the Hille-Yosida condition on $X_0$, one needs to extend the mentioned projections from $X_0$ to $X$ and outline the associated statement below.  
\begin{lemma}\cite[Proposition 3.5]{PM_2009}\label{lem2.11}
	Let $\Pi_0:X_0\to{X_0}$ be a bounded linear projection satisfying
\begin{equation*}
	\Pi_0{\left(\lambda{I}-A_0\right)}^{-1}={\left(\lambda{I}-A_0\right)}^{-1}\Pi_0,~\forall{\lambda>\vartheta}.
\end{equation*}
and
\begin{equation*}
	\Pi_0X_0\subset{D\left(A_0\right)}, ~\text{and}~{A_0|}_{\Pi_0X_0}~\text{is bounded}.
\end{equation*}
Then there exists a unique bounded projection $\Pi:X\to{X}$ such that
\begin{itemize}
	\item [$\left(\mathrm{i}\right)$]${\Pi|}_{X_0}=\Pi_0$;
	
	\item [$\left(\mathrm{ii}\right)$]$\Pi\left(X\right)\subset{X_0}$;
	
	\item [$\left(\mathrm{iii}\right)$]$\Pi{\left(\lambda{I}-A\right)}^{-1}={\left(\lambda{I}-A\right)}^{-1}\Pi$,~$\forall{\lambda>\vartheta}$.
\end{itemize}
\end{lemma}

With the mentioned information at hand, we are ready to state the decomposition on $X$. For each $k\in\left\{c,u,cu\right\}$, denote by $\Pi_k:X\to{X}$ the unique extension of $\Pi_{0k}$ and let $\Pi_s=I_X-\left(\Pi_c+\Pi_u\right)$. Then we have for each ${p\in\left\{c,u,s,cu\right\}}$, $\Pi_p{\left(\lambda{I}-A\right)}^{-1}={\left(\lambda{I}-A\right)}^{-1}\Pi_p,~\forall{\lambda>\vartheta}$, and $\Pi_p\left(X\right)\subset{X_0}$. Also, $\Pi_p\left(X_0\right)=\Pi_{0p}\left(X_0\right)=X_{0p}$. Denote
	$X_p=\Pi_p\left(X\right),~\text{and}~A_p=A|_{X_p}$. Then for each $k\in\left\{c,u\right\}$, we have $X_{0k}=X_k,~X_0=X_{0s}\oplus{X_{0c}}\oplus{X_{0u}},~\text{and}~X=X_{s}\oplus{X_{c}}\oplus{X_{u}}$.
Herein, $X_{0c}$, $X_{0u}$ and $X_{0s}$ are called center subspace, unstable subspace and stable subspace of $X_0$ respectively.  
For more details, we refer to\cite[page 22]{PM_2009}.

\section{Smooth Center-unstable manifolds}\label{sec3}
This section is dedicated to the existence of invariant manifolds of \eqref{eq1} through the Lyapunov-Perron method. 
\begin{definition}
	Let $\mathcal{M}=\{\mathcal{M}\left(\omega\right):\omega\in \Omega\}$ be a family of $X_0$, then $\mathcal M$ is called a $C^k$ random center-unstable invariant manifold of the random dynamical system $\varphi$ if $\mathcal M$ is invariant under $\varphi$ and given by the graph of a Lipschitz continuous map $h^{cu}\left(\cdot,\omega\right):X_{0cu}\to X_{0s}$ which is $C^k$ in $\xi$, $k\geq1$, i.e.,
	\begin{equation*}
		\mathcal M\left(\omega\right)=\{\xi+h^{cu}\left(\xi,\omega\right):\xi\in X_{0cu}\}.
	\end{equation*}
We denote $\mathcal M\left(\omega\right)$ by $\mathcal M^{cu}\left(\omega\right)$. Accordingly, a $C^k$ random center(resp. center-stable) invariant manifold $\mathcal M^c\left(\omega\right)$(resp. $\mathcal M^{cs}\left(\omega\right)$) is given by the graph of a Lipschitz map $h^i\left(\cdot,\omega\right):X_{0i}\to X_{0j}$, $i=c$(resp. $cs$) and $j=su$(resp. $j=u$).
\end{definition}
We extract some results from \cite{LZ2021} about integrated solutions and invariant manifolds of \eqref{eq4}.

Set
	\begin{equation*}
	\left(S\diamond{G_x\left(v\right)}\right)\left(t\right)=\lim_{\lambda\to+\infty}\int_0^t\phi_{A_0}\left(t,s\right)\lambda{\left(\lambda{I}-A\right)}^{-1}G\left(\theta_s\omega,v\left(s,\omega,x\right)\right)ds.
\end{equation*}
We have the following result:
\begin{proposition}\cite[Proposition 2.1]{LZ2021}\label{pro3.2}
Under Assumption \ref{as2.2}, We have
\begin{equation}\label{eq8}
\Pi_{0s}\left(S\diamond{G_x\left(v\right)}\right)\left(t\right)=\lim_{\lambda\to+\infty}\int_0^t\phi_{A_{0s}}\left(t,s\right)\lambda{\left(\lambda{I}-A_s\right)}^{-1}\Pi_sG\left(\theta_s\omega,v\left(s,\omega,x\right)\right)ds,	
\end{equation}
while for each $k\in\{c,u\}$, 
\begin{equation}\label{eq9}
\Pi_{0k}\left(S\diamond{G_x\left(v\right)}\right)\left(t\right)=\int_0^t\phi_{A_{0k}}\left(t,s\right)\Pi_k G\left(\theta_s\omega,v\left(s,\omega,x\right)\right)ds.		
\end{equation}
Moreover, for each $\zeta>-\beta$, there exists $C_\zeta>0$, such that for $\forall{t\geq0}$, we obtain
\begin{equation}\label{eq10}
	\Vert\Pi_{0s}\left(S\diamond{G_x\left(v\right)}\right)\left(t\right)\Vert\leq{C_\zeta}\underset{s\in[0,t]}{\sup}e^{\zeta\left(t-s\right)}e^{\int_s^t{z\left(\theta_r\omega\right)dr}}\left\Vert{G\left(\theta_s\omega,v\left(s,\omega,x\right)\right)}\right\Vert.
\end{equation}
\end{proposition}
\begin{lemma}\cite[Lemma 2.5, Theorem 3.1]{LZ2021}\label{lem3.3}
	The following assertions are valid:
	\begin{itemize}
		\item [$\left(\mathrm{i}\right)$]Under Assumption \ref{as2.1}, \eqref{eq4} possesses a unique integrated solution on $X_0$ given by 
		\begin{equation*}
			v\left(t\right)=\phi_{A_0}\left(t,0\right)x+\left(S\diamond{G_x\left(v\right)}\right)\left(t\right),
		\end{equation*}
		which is the so-called variation of constants fomula and generates a random dynamical system $\Phi$. 	Moreover, 
\begin{equation*}
	\left(t,\omega,x\right)\to{\Xi}^{-1}\left(\theta_t\omega,\Phi\left(t,\omega,\Xi\left(\omega,x\right)\right)\right)=:\hat{\Phi}\left(t,\omega,x\right),
\end{equation*}	
is  a random dynamical system. Also, this process is a solution of \eqref{eq1} for $x\in X_0$.

		\item [$\left(\mathrm{ii}\right)$]Under Assumptions \ref{as2.1}-\ref{as2.2}, if $\gamma<\eta<\min\left\{\alpha,\beta\right\}$ and $-\beta<\zeta<-\eta$ with
\begin{equation*}
	KL\left(C_\zeta+\frac1{\eta-\gamma}+\frac1{\alpha-\eta}\right)<1,
\end{equation*}
then there exists a center invariant manifold for \eqref{eq4}, given by
\begin{equation*}
	\mathcal{M}^c\left(\omega\right)=\{\xi+h^c\left(\xi,\omega\right):\xi\in{X_{0c}}\},
\end{equation*} 
where $h^c\left(\cdot,\omega\right):X_{0c}\to{X_{0us}}$ is a Lipschitz continuous map with $h^c\left(0,\omega\right)=0$ and 
\begin{equation*}
	Lip~h^c\left(\cdot,\omega\right)\leq\frac{K\left(KL+C_\zeta{L}\left(\alpha-\eta\right)\right)}{\left(\alpha-\eta\right)\left(1-KL\left(C_\zeta+\frac1{\eta-\gamma}+\frac1{\alpha-\eta}\right)\right)}.
	\end{equation*}
	\end{itemize}
\end{lemma}

For $\eta_{cu}\in\left(-\beta,-\gamma\right)$, define the Banach space
\begin{equation*}
	\mathcal{C}_{\eta_{cu}}\left((-\infty,0];X_0\right)=\left\{f\in{C\left((-\infty,0],X_0\right)}:\underset{t\leq0}{\sup}~e^{-\eta_{cu} t-\int_0^tz\left(\theta_s\omega\right)ds}{\left\Vert{f\left(t\right)}\right\Vert}<\infty\right\},
\end{equation*} 
equipped with the norm
\begin{equation*}
	\Vert{f\left(t\right)}\Vert_{\mathcal{C}_{\eta_{cu}}}=\underset{t\leq0}{\sup}~e^{-\eta_{cu} t-\int_0^tz\left(\theta_s\omega\right)ds}{\Vert{f\left(t\right)}\Vert}.
\end{equation*}
And for $\eta_{cs}\in\left(\gamma,\min\{\alpha,\beta\}\right)$, define the Banach space
\begin{equation*}
	\mathcal{C}_{\eta_{cs}}\left([0,+\infty);X_0\right)=\left\{f\in{C\left([0,+\infty),X_0\right)}:\underset{t\geq0}{\sup}~e^{-\eta_{cs} t-\int_0^tz\left(\theta_s\omega\right)ds}{\left\Vert{f\left(t\right)}\right\Vert}<\infty\right\},
\end{equation*} 
equipped with the norm
\begin{equation*}
	\Vert{f\left(t\right)}\Vert_{\mathcal{C}_{\eta_{cs}}}=\underset{t\geq0}{\sup}~e^{-\eta_{cs}t-\int_0^tz\left(\theta_s\omega\right)ds}{\Vert{f\left(t\right)}\Vert}.
\end{equation*}
Define
\begin{equation*}
	\mathcal{M}^{cu}\left(\omega\right)=\{x\in X_0:v\left(t,\omega,x\right)\in\mathcal{C}_{\eta_{cu}}\left((-\infty,0];X_0\right)\}.
\end{equation*}
\begin{equation*}
	\mathcal{M}^{cs}\left(\omega\right)=\{x\in X_0:v\left(t,\omega,x\right)\in\mathcal{C}_{\eta_{cs}}\left([0,+\infty);X_0\right)\}.
\end{equation*}
where $v\left(t,\omega,x\right)$ is the solution of \eqref{eq4} with initial data $v\left(0\right)=x$. Let Assumptions \ref{as2.1}-\ref{as2.2} be satisfied from now on.
\begin{theorem}\label{thm3.4}
	For $-\beta<\zeta<\eta_{cu}<-\gamma$ with
\begin{equation}\label{eq11}
	{KL\left(C_\zeta-\frac1{\gamma+\eta_{cu}}-\frac1{\eta_{cu}-\alpha}\right)}<1,
\end{equation}
there exists a Lipschitz invariant center-unstable manifold for \eqref{eq4} which is given by 
\begin{equation*}
	\mathcal{M}^{cu}\left(\omega\right)=\{\xi+h^{cu}\left(\xi,\omega\right):\xi\in X_{0cu}\},
\end{equation*}
where $h^{cu}\left(\cdot,\omega\right):X_{0cu}\to X_{0s}$ is Lipschitz continuous with $h^{cu}\left(0,\omega\right)=0$ and 
\begin{equation*}
	Lip~h^{cu}\left(\cdot,\omega\right)\leq K_u=\frac{KLC_\zeta}{1-KL\left(C_\zeta-\frac1{\gamma+\eta_{cu}}-\frac1{\eta_{cu}-\alpha}\right)}.
\end{equation*}

\end{theorem}
\begin{proof}We proceed it in three steps.

\noindent {\bf Step 1.} 
	We prove that $x\in{\mathcal{M}^{cu}\left(\omega\right)}$ if and only if there exists a function $v\left(\cdot,\omega,x\right)\in{\mathcal C_{\eta_{cu}}\left((-\infty,0];X_0\right)}$ with $v\left(0\right)=x$ which satisfies
\begin{equation}\label{eq12}
\begin{split}
v\left(t,\omega,x\right)=&\phi_{A_{0c}}\left(t,0\right)\xi+\int_0^t\phi_{A_{0c}}\left(t,s\right)\Pi_cG\left(\theta_s\omega,v\left(s\right)\right)ds\\&+\int_0^t\phi_{A_{0u}}\left(t,s\right)\Pi_uG\left(\theta_s\omega,v\left(s\right)\right)ds\\&+\lim_{\lambda\to+\infty}\int_{-\infty}^t\phi_{A_{0s}}\left(t,s\right)\lambda\left(\lambda{I}-A_s\right)^{-1}\Pi_sG\left(\theta_s\omega,v\left(s\right)\right)ds, 
\end{split}	
\end{equation}
where $\xi=\Pi_{0cu}x$.

To this aim, we let $x\in{\mathcal{M}^{cu}\left(\omega\right)}$ and assume that $v\in{\mathcal C_{\eta_{cu}}\left((-\infty,0];X_0\right)}$ is a solution of \eqref{eq4}. 
By the variation of constants formula from Lemma \ref{lem3.3}(i), one gets
\begin{equation}\label{eq13}
\begin{split}
	\Pi_{0cu}v\left(t,\omega,x\right)=&\phi_{A_{0c}}\left(t,0\right)\xi+\int_0^t\phi_{A_{0c}}\left(t,s\right)\Pi_c  G\left(\theta_s\omega,v\left(s\right)\right)ds\\&+\int_0^t\phi_{A_{0u}}\left(t,s\right)\Pi_u G\left(\theta_s\omega,v\left(s\right)\right)ds.
	\end{split}
\end{equation}
Moreover, for all $t,l\in\mathbb{R}$ with $t\geq{l}$,
\begin{equation}\label{eq14}
\begin{split}
	\Pi_{0s}v\left(t\right)&=\phi_{A_{0s}}(t,l)\Pi_{0s}v(l)+\Pi_{0s}\left(S\diamond{G}_x\left(v\left(l+\cdot\right)\right)\right)(t-l)\\&=\phi_{A_{0s}}\left(t,l\right)\Pi_{0s}v\left(l\right)+\lim_{\lambda\to+\infty}\int_l^t\phi_{A_{0s}}\left(t,s\right)\lambda\left(\lambda{I}-A_s\right)^{-1}\\&\quad\times\Pi_sG\left(\theta_s\omega,v\left(s\right)\right)ds.
\end{split}
\end{equation}
According to \eqref{eq7}, for $l\leq0$,
\begin{equation*}
	\begin{split}
	\left\Vert\phi_{A_{0s}}\left(t,l\right)\Pi_{0s}v\left(l\right)\right\Vert\leq&{K}e^{-\beta\left(t-l\right)+\int_l^tz\left(\theta_r\omega\right)dr}\left\Vert{v\left(l\right)}\right\Vert\\\leq&{K}e^{-\beta\left(t-l\right)-\eta_{cu}{l}+\int_0^t{z\left(\theta_r\omega\right)dr}}e^{\eta_{cu}{l}-\int_0^l{z\left(\theta_r\omega\right)dr}}\left\Vert{v\left(l\right)}\right\Vert\\\leq&{K}e^{\left(\beta+\eta_{cu}\right)l-\beta{t}+\int_0^t{z\left(\theta_r\omega\right)dr}}{\left\Vert{v}\right\Vert}_{\mathcal C_{\eta_{cu}}}.	
	\end{split}
\end{equation*}
Let $l\to-\infty$, since $\beta+\eta_{cu}>0$, we have
\begin{equation}\label{eq15}
	\Pi_{0s}v\left(t,\omega,x\right)=\lim_{\lambda\to+\infty}\int_{-\infty}^t\phi_{A_{0s}}\left(t,s\right)\lambda\left(\lambda{I}-A_s\right)^{-1}\Pi_s{G\left(\theta_s\omega,v\left(s\right)\right)}ds.
\end{equation}
By letting $l\to -\infty$ and combining with  \eqref{eq13}, we get \eqref{eq12}. The converse can be obtained by a straight computation.

\noindent {\bf Step 2.}
We claim that for every $\xi\in X_{0cu}$, \eqref{eq12} has a unique solution in $\mathcal{C}_{\eta_{cu}}\left((-\infty,0];X_0\right)$. 

To show this claim, we denote $\mathcal{J}\left(v,\omega,\xi\right)$ the right side of \eqref{eq12}, i.e.,
\begin{equation}\label{eq16}
\begin{split}
	\mathcal{J}\left(v,\omega,\xi\right)=&\phi_{A_{0c}}\left(t,0\right)\xi+\int_0^t\phi_{A_{0c}}\left(t,s\right)\Pi_c G\left(\theta_s\omega,v\left(s\right)\right)ds\\&+\int_0^t\phi_{A_{0u}}\left(t,s\right)\Pi_u G\left(\theta_s\omega,v\left(s\right)\right)ds\\&+\lim_{\lambda\to+\infty}\int_{-\infty}^t\phi_{A_{0s}}\left(t,s\right)\lambda\left(\lambda{I}-A_s\right)^{-1}\Pi_sG\left(\theta_s\omega,v\left(s\right)\right)ds, \end{split}
\end{equation}

It needs to show that $\mathcal{J}$ maps from $\mathcal{C}_{\eta_{cu}}\left((-\infty,0];X_0\right)\times{X_{0cu}}$ to $\mathcal{C}_{\eta_{cu}}\left((-\infty,0];X_0\right)$ and it is a uniform contraction. For $v,\tilde{v}\in{\mathcal{C}_{\eta_{cu}}\left((-\infty,0];X_0\right)}$ and $t\leq0$,
\begin{equation*}
\begin{split}
	&{\Big \Vert{\mathcal{J}\left(v,\omega,\xi\right)-\mathcal{J}\left(\tilde{v},\omega,\xi\right)}\Big\Vert}_{\mathcal{C}_{\eta_{cu}}}=\underset{t\leq0}{\sup}~e^{-\eta_{cu}t-\int_0^tz\left(\theta_s\omega\right)ds}{\Big\Vert{\mathcal{J}\left(v,\omega,\xi\right)-\mathcal{J}\left(\tilde{v},\omega,\xi\right)}\Big\Vert}\\&\leq KL\sup_{t\leq0}\int_t^0e^{\left(\gamma+\eta_{cu}\right)\left(s-t\right)-\eta_{cu}s-\int_0^sz\left(\theta_r\omega\right)dr}\left\Vert v-\tilde v\right\Vert ds\\&+KL\sup_{t\leq0}\int_t^0e^{\left(\eta_{cu}-\alpha\right)\left(s-t\right)-\eta_{cu}s-\int_0^sz\left(\theta_r\omega\right)dr}\left\Vert v-\tilde v\right\Vert ds+\underset{t\leq0}{\sup}~e^{-\eta_{cu}t-\int_0^tz\left(\theta_s\omega\right)ds}\Delta\mathcal J_s\\&\leq KL\left(-\frac{1}{\gamma+\eta_{cu}}-\frac{1}{\eta_{cu}-\alpha}\right)\left\Vert v-\tilde v\right\Vert_{\mathcal C_{\eta_{cu}}}+\underset{t\leq0}{\sup}~e^{-\eta_{cu}t-\int_0^tz\left(\theta_s\omega\right)ds}\Delta\mathcal J_s.
	\end{split}
\end{equation*}
where
\begin{equation*}
\begin{split}
\Delta\mathcal J_s=\left\Vert\lim_{\lambda\to+\infty}\int_{-\infty}^t\phi_{A_{0s}}\left(t,s\right)\lambda\left(\lambda{I}-A_s\right)^{-1}\Pi_s\Big(G\left(\theta_s\omega,v\left(s\right)\right)-G\left(\theta_s\omega,\tilde{v}\left(s\right)\right)\Big)ds\right\Vert.
\end{split}
\end{equation*}
%
%
%We divide ${\Vert{J\left(v,\xi\right)-J\left(\tilde{v},\xi\right)}\Vert}_{C_\eta\left(\mathbb{R},X_0\right)}$ into three parts. 
According to Proposition \ref{pro3.2}, for $\zeta$ with $-\beta<\zeta<\eta_{cu}$, we have
\begin{equation*}
	\begin{split}
	\Delta\mathcal J_s&=\bigg\Vert\lim_{\lambda\to{+\infty}}\lim_{r\to{-\infty}}\int_r^t\phi_{A_{0s}}\left(t,s\right)\lambda\left(\lambda{I}-A_s\right)^{-1}\\&\times\Pi_s\Big(G\left(\theta_s\omega,v\left(s\right)\right)-G(\theta_s\omega,\tilde v\left(s\right))\Big)ds\bigg\Vert\\&\quad=\bigg\Vert\lim_{\lambda\to{+\infty}}\lim_{r\to{-\infty}}\int_0^{t-r}\phi_{A_{0s}}\left(t,l+r\right)\lambda\left(\lambda{I}-A_s\right)^{-1}\\
		&\qquad\times\Pi_s\Big(G\left(\theta_{l+r}\omega,v\left(l+r\right)\right)-G(\theta_{l+r}\omega,\tilde v\left(l+r\right))\Big)dl\bigg\Vert\\&\quad\leq\lim_{r\to{-\infty}}LC_\zeta\underset{l\in[0,t-r]}{\sup}{e^{\zeta\left(t-r-l\right)}}e^{\int_{l+r}^t{z\left(\theta_\tau\omega\right)}d\tau}\Vert{v\left(l+r\right)-\tilde v\left(l+r\right)}\Vert\\&\quad=LC_\zeta\underset{\sigma\in(-\infty,t]}{\sup}e^{\zeta\left(t-\sigma\right)}e^{\int_\sigma^t{z\left(\theta_\tau\omega\right)}d\tau}\Vert{v\left(\sigma\right)-\tilde v\left(\sigma\right)}\Vert,
	\end{split}
\end{equation*}
implying that
\begin{equation*}
\begin{split}
	\underset{t\leq0}{\sup}~e^{-\eta_{cu}t-\int_0^tz\left(\theta_s\omega\right)ds}\Delta\mathcal{J}_s&\leq{LC_\zeta}\underset{t\leq0}{\sup}~e^{-\eta_{cu}t}\underset{\sigma\in(-\infty,t]}{\sup}e^{\eta_{cu}\sigma+\zeta\left(t-\sigma\right)}\\&\quad{\times}e^{-\eta_{cu}\sigma-\int_0^\sigma{z\left(\theta_\tau\omega\right)}d\tau}\Vert{v\left(\sigma\right)-\tilde v\left(\sigma\right)}\Vert\\&\leq{LC_\zeta}\underset{t\leq0}{\sup}\underset{\sigma\in(-\infty,t]}{\sup}e^{\left(\eta_{cu}-\zeta\right)\sigma+\left(\zeta-\eta_{cu}\right)t}\Vert{v-\tilde v}\Vert_{\mathcal{C}_{\eta_{cu}}}\\&\leq LC_\zeta\Vert{v-\tilde v}\Vert_{\mathcal{C}_{\eta_{cu}}}.
\end{split}
\end{equation*}
By combining the above estimates, we have
\begin{equation*}
	{\left\Vert{\mathcal{J}\left(v,\omega,\xi\right)-\mathcal{J}\left(\tilde{v},\omega,\xi\right)}\right\Vert}_{\mathcal{C}_{\eta_{cu}}}\leq{KL\left(C_\zeta-\frac1{\gamma+\eta_{cu}}-\frac1{\eta_{cu}-\alpha}\right)}{\Vert{v-\tilde{v}}\Vert}_{\mathcal{C}_{\eta_{cu}}}.
\end{equation*}
By \eqref{eq11}, we obtain that $\mathcal{J}$ is a uniform contraction with respect to $v$. Furthermore, $\mathcal{J}\left(\cdot,\omega,\xi\right)$ is well defined from ${\mathcal C_{\eta_{cu}}\left((-\infty,0];X_0\right)}\times{X_{0cu}}$ to ${\mathcal C_{\eta_{cu}}\left((-\infty,0];X_0\right)}$ by setting $\tilde{v}=0$. From the  contraction mapping principle, for any given $\xi\in{X_{0cu}}$, $\mathcal{J}\left(\cdot,\omega,\xi\right)$ has a unique fixed point $\overline{v}\in{\mathcal{C}_{\eta_{cu}}\left((-\infty,0];X_0\right)}$. That is, for each $t\leq0$, $\xi\in X_{0cu}$, $\overline{v}\left(t,\omega,\xi\right)$ is a unique solution to \eqref{eq12} and 
\begin{equation*}
	\mathcal{J}\left(\overline{v},\omega,\xi\right)=\overline{v}\left(t,\omega,\xi\right).
\end{equation*} 
Also, for $\forall{\xi_1,\xi_2\in{X_{0cu
}}}$, 
\begin{equation*}	
\begin{split}
	&{\left\Vert{\mathcal{J}\left(\overline{v},\omega,\xi_1\right)-\mathcal{J}\left(\overline{v},\omega,\xi_2\right)}\right\Vert}_{\mathcal{C}_{\eta_{cu}}}\\&\quad={\left\Vert{\overline{v}\left(t,\omega,\xi_1\right)-\overline{v}\left(t,\omega,\xi_2\right)}\right\Vert}_{\mathcal{C}_{\eta_{cu}}}\\&\quad\leq{K}\left\Vert{\xi_1-\xi_2}\right\Vert+{KL\left(C_\zeta-\frac1{\gamma+\eta_{cu}}-\frac1{\eta_{cu}-\alpha}\right)}{\left\Vert{\overline{v}\left(t,\omega,\xi_1\right)-\overline{v}\left(t,\omega,\xi_2\right)}\right\Vert}_{\mathcal{C}_{\eta_{cu}}}.
	\end{split}
\end{equation*}  
Thus
\begin{equation}\label{eq17}
	{\left\Vert{\overline{v}\left(t,\omega,\xi_1\right)-\overline{v}\left(t,\omega,\xi_2\right)}\right\Vert}_{\mathcal{C}_{\eta_{cu}}}\leq\frac{K}{1-KL\left(C_\zeta-\frac1{\gamma+\eta_{cu}}-\frac1{\eta_{cu}-\alpha}\right)}\left\Vert{\xi_1-\xi_2}\right\Vert.
\end{equation}
So $\overline{v}\left(t,\omega,\xi\right)$ is Lipschitz continuous from $X_{0cu}$ to $\mathcal{C}_{\eta_{cu}}\left((-
\infty,0];X_0\right)$. Since $\overline{v}\left(\cdot,\omega,\xi\right)$ can be an $\omega$-wise limit of the iteration of a contraction mapping $\mathcal{J}$ starting at 0 and $\mathcal{J}$  maps a measurable function to a measurable function, $\overline{v}\left(\cdot,\omega,\xi\right)$ is measurable. By \cite[Lemma III.14]{Cas1977}, $\overline{v}$ is measurable with respect to $\left(\cdot,\omega,\xi\right)$.

\noindent {\bf Step 3.} We will prove that the center-unstable manifold is given by the graph of a Lipschitz continuous map and is invariant.

Let $h^{cu}\left(\xi,\omega\right)=\Pi_{0s}{\overline{v}\left(0,\omega,\xi\right)}$, then by \eqref{eq15},
\begin{equation*}
\begin{split}
	h^{cu}\left(\xi,\omega\right)=&\lim_{\lambda\to+\infty}\int_{-\infty}^0\phi_{A_{0s}}\left(0,s\right)\lambda{\left(\lambda{I}-A_s\right)}^{-1}\Pi_sG\left(\theta_s\omega,\overline{v}\left(s,\omega,\xi\right)\right)ds.
\end{split}	
\end{equation*}
Since $\overline{v}\left(0,\omega,x\right)=x$, we have $h^{cu}\left(0,\omega\right)=0$ and thus $h^{cu}$ is $\mathcal{F}$-measurable. Indeed, by the estimation of $\Delta\mathcal J_s$ and \eqref{eq17}, it follows that
\begin{equation*}
	\left\Vert{h^{cu}\left(\xi_1,\omega\right)-h^{cu}\left(\xi_2,\omega\right)}\right\Vert\leq\frac{KLC_\zeta}{1-KL\left(C_\zeta-\frac1{\gamma+\eta_{cu}}-\frac1{\eta_{cu}-\alpha}\right)}\left\Vert{\xi_1-\xi_2}\right\Vert.
\end{equation*}
From the definition of $h^{cu}\left(\xi,\omega\right)$ and the claim in {\bf Step 1}, it follows that $x\in{\mathcal{M}^{cu}\left(\omega\right)}$ if and only if there exists $\xi\in{X_{0cu}}$ such that $x=\xi+h^{cu}\left(\xi,\omega\right)$. Therefore, we have
\begin{equation*}
		\mathcal{M}^{cu}\left(\omega\right)=\{\xi+h^{cu}\left(\xi,\omega\right):\xi\in{X_{0cu}}\}.
\end{equation*}
By Definition \ref{def2.2}, we need to show that for each $x\in{\mathcal M^{cu}\left(\omega\right)}$, $v\left(r,\omega,x\right)\in{\mathcal M^{cu}\left(\omega\right)}$ for all $r\geq0$. By the cocycle property, we have $v\left(t+r,\omega,x\right)=v\left(t,\theta_r\omega,v\left(r,\omega,x\right)\right)\in{\mathcal{C}_{\eta_{cu}}\left((-\infty,0];X_0\right)}$.By the definition of $\mathcal M^{cu}\left(\omega\right)$, we have $v\left(r,\omega,x\right)\in{\mathcal M^{cu}\left(\theta_r\omega\right)}$. 

\end{proof}

\begin{theorem}\label{thm3.5}
Let $v\left(t,\omega,x\right)$ be a solution of \eqref{eq4} and $u\left(t,\omega,x\right)$ be the solution of \eqref{eq1}, then ${\mathcal M^{cu*}\left(\omega\right)}=\Xi^{-1}\left(\omega,\mathcal M^{cu}\left(\omega\right)\right)$ is a center-unstable invariant manifold of \eqref{eq1}. 	
\end{theorem}
\begin{proof}
In fact, it follows from Lemma \ref{lem3.3}(i) that
\begin{equation*}
\begin{split}
	u\left(t,\omega,\mathcal M^{cu*}\left(\omega\right)\right)&=\Xi^{-1}\left(\theta_t\omega,v\left(t,\omega,\Xi\left(\omega,\mathcal M^{cu*}\left(\omega\right)\right)\right)\right)\\&=\Xi^{-1}\left(\theta_t\omega,v\left(t,\omega,\mathcal M^{cu}\left(\omega\right)\right)\right)\subset{\Xi}^{-1}\left(\theta_t\omega,\mathcal M^{cu}\left(\theta_t\omega\right)\right)\\&=\mathcal M^{cu*}\left(\theta_t\omega\right).
\end{split}
\end{equation*}	
So $\mathcal M^{cu*}\left({\omega}\right)$ is invariant. Moreover,
\begin{equation*}
	\begin{split}
		\mathcal M^{cu*}\left(\omega\right)&=\Xi^{-1}\left(\omega,\mathcal M^{cu}\left(\omega\right)\right)\\&=\left\{u_0=\Xi^{-1}\left(\omega,\xi+h^{cu}\left(\xi,\omega\right)\right):\xi\in{X_{0cu}}\right\}\\&=\left\{u_0=e^{z\left(\theta_t\omega\right)}\left(\xi+h^{cu}\left(\xi,\omega\right)\right):\xi\in{X_{0cu}}\right\}\\&=\left\{u_0=\xi+e^{z\left(\theta_t\omega\right)}h^{cu}\left(e^{-z\left(\theta_t\omega\right)}\xi,\omega\right):\xi\in{X_{0cu}}\right\}.
	\end{split}
\end{equation*}
 Therefore, $\mathcal M^{cu*}\left(\omega\right)$ is a Lipschitz center-unstable invariant manifold given by the graph of a Lipschitz continuous function $h^{cu*}\left(\xi,\omega\right)=e^{z\left(\theta_t\omega\right)}h^{cu}\left(e^{-z\left(\theta_t\omega\right)}\xi,\omega\right)$ over $X_{0cu}$.
\end{proof}

We prove that $\mathcal{M}^{cu}\left(\omega\right)$ is a $C^k$ invariant center-unstable manifold for \eqref{eq4} by induction.

\begin{theorem}\label{thm3.6}
Assume $F$ is $C^k$ in $u$. If $-\beta<\zeta<k\eta_{cu}<-\gamma$ with
\begin{equation}\label{eq18}
	KL\left(C_{\zeta+\left(i-1\right)\eta_{cu}}-\frac1{\gamma+i\eta_{cu}}-\frac1{i\eta_{cu}-\alpha}\right)<1, ~\text{for all} ~1\leq{i}\leq{k},
\end{equation}	 
then $\mathcal{M}^{cu}\left(\omega\right)$ is a $C^k$ center-unstable invariant manifold of \eqref{eq4}, i.e., $h^{cu}\left(\xi,\omega\right)$ is $C^k$ in $\xi$.
\end{theorem}
\begin{proof}
	\noindent {\bf Step 1.} 	We prove that the fixed point $\overline{v}\left(t,\omega,\xi\right)$ of the Lyapunov-Perron operator $\mathcal J\left(\cdot,\omega,\xi\right)$ defined in Theorem \ref{thm3.4} is $C^1$. For notational simplify, we denote $\mathcal{C}_{\eta_{cu}}\left((-\infty,0];X_0\right)$ by $\mathcal{C}_{\eta_{cu}}$ from now on. Notice that for $-\beta<\eta^{'}_{cu}<\eta^{''}_{cu}<-\gamma$, $\mathcal C_{\eta^{''}_{cu}}\subset{\mathcal C_{\eta^{'}_{cu}}}$, thus a fixed point in $\mathcal C_{\eta^{''}_{cu}}$ must be in $\mathcal C_{\eta^{'}_{cu}}$. So $\overline{v}\left(\cdot,\omega,\xi\right)$ is independent of $\eta_{cu}$ satisfying \eqref{eq18}. 

Clearly, there exits $\nu^*>0$ such that for $0<\nu\leq\nu^*$, we have $-\beta<\zeta<i\eta_{cu}+\nu<i\eta_{cu}+2\nu<-\gamma$ and
\begin{equation}\label{eq19}
	KL\left(C_{\zeta-j\nu}-\frac1{\gamma+i\eta_{cu}+\nu}-\frac1{i\eta_{cu}-\alpha+\nu}\right)<1,~\text{for}~\text{all}~1\leq{i}\leq{k},~j=1,2.
\end{equation}
Then $\mathcal J\left(\cdot,\omega,\xi\right)$ is a uniform contraction in $\mathcal C_{\eta_{cu}+\nu}\subset{\mathcal C_{\eta_{cu}}}$. By \eqref{eq16}, for $\xi\in{X_{0cu}}$,
\begin{equation*}
\begin{split}
	\overline v\left(t,\omega,\xi\right)=&\phi_{A_{0c}}\left(t,0\right)\xi+\int_0^t\phi_{A_{0c}}\left(t,s\right)\Pi_c G\left(\theta_s\omega,\overline v\left(s,\omega,\xi\right)\right)ds\\&+\int_0^t\phi_{A_{0u}}\left(t,s\right)\Pi_u G\left(\theta_s\omega,\overline v\left(s,\omega,\xi\right)\right)ds\\&+\lim_{\lambda\to+\infty}\int_{-\infty}^t\phi_{A_{0s}}\left(t,s\right)\lambda\left(\lambda{I}-A_s\right)^{-1}\Pi_s G\left(\theta_s\omega,\overline v\left(s,\omega,\xi\right)\right)ds.\end{split}
\end{equation*}

By \eqref{eq17}, for $\forall\xi,\xi_0\in{X_{0cu}}$ and $i=1,2$, $\overline{v}$ satisfies
\begin{equation}\label{eq20}
	{\left\Vert{\overline{v}\left(t,\omega,\xi\right)-\overline{v}\left(t,\omega,\xi_0\right)}\right\Vert}_{\mathcal C_{\eta_{cu}+i\nu}}\leq\frac{K}{1-KL\left(C_{\zeta-i\nu}-\frac1{\gamma+\eta_{cu}+i\nu}-\frac1{\eta_{cu}-\alpha+i\nu}\right)}\left\Vert{\xi-\xi_0}\right\Vert.
\end{equation}	 

Fix $\nu\in(0,\nu^*]$, we first prove that $\overline{v}$ is differentialble from $X_{0cu}$ to $\mathcal C_{\eta_{cu}-\nu}$. For $\xi_0\in{X_{0cu}}$, $v\in{\mathcal C_{\eta_{cu}+\nu}}$, define $\mathcal{H}:\mathcal C_{\eta_{cu}+\nu}\to{\mathcal C_{\eta_{cu}+\nu}}$ as follows
\begin{equation*}
\begin{split}
	\mathcal{H}v=&\int_0^t\phi_{A_{0c}}\left(t,s\right)\Pi_cD_vG\left(\theta_s\omega,\overline{v}\left(s,\omega,\xi_0\right)\right)v\left(s\right)ds\\&+\int_0^{t}\phi_{A_{0u}}\left(t,s\right)\Pi_uD_vG\left(\theta_s\omega,\overline{v}\left(s,\omega,\xi_0\right)\right)v\left(s\right)ds\\&+\lim_{\lambda\to+\infty}\int_{-\infty}^t\phi_{A_{0s}}\left(t,s\right)\lambda\left(\lambda{I}-A_s\right)^{-1}\Pi_s D_vG\left(\theta_s\omega,\overline{v}\left(s,\omega,\xi_0\right)\right)v\left(s\right)ds.
\end{split}
\end{equation*}
Using the same arguments as we proved that $\mathcal J$ is a contraction and noticing that $\left\Vert{D_v G\left(\theta_s\omega,v\right)}\right\Vert\leq{Lip~G=L}$, we have $\mathcal{H}$ is a bounded linear operator from $\mathcal C_{\eta_{cu}+\nu}$ to itself with the norm
\begin{equation*}
	\left\Vert{\mathcal{H}}\right\Vert\leq{K}L\left(C_{\zeta-\nu}-\frac1{\gamma+\eta_{cu}+\nu}-\frac1{\eta_{cu}-\alpha+\nu}\right)<1.
\end{equation*}
Thus $id-\mathcal{H}$ is invertible in $\mathcal C_{\eta_{cu}+\nu}$. For $\xi,\xi_0\in{X_{0cu}}$, set
\begin{equation*}
	\begin{split}
		\mathcal I\left(t\right)=&\int_0^t\phi_{A_{0c}}\left(t,s\right)\Pi_c\overline\Delta_G\left(s\right)ds+\int_{0}^{t}\phi_{A_{0u}}\left(t,s\right)\Pi_u\overline\Delta_G\left(s\right)ds\\&+\lim_{\lambda\to+\infty}\int_{-\infty}^t\phi_{A_{0s}}\left(t,s\right)\lambda\left(\lambda{I}-A_s\right)^{-1}\Pi_s\overline\Delta_G\left(s\right)ds.
	\end{split}
\end{equation*}
where 
\begin{equation*}\begin{split}
	\overline\Delta_G\left(s\right)=&G\left(\theta_s\omega,\overline{v}\left(s,\omega,\xi\right)\right)- G\left(\theta_s\omega,\overline{v}\left(s,\omega,\xi_0\right)\right)\\&-D_v G\left(\theta_s\omega,\overline{v}\left(s,\omega,\xi_0\right)\right)\left(\overline{v}\left(s,\omega,\xi\right)-\overline{v}\left(s,\omega,\xi_0\right)\right).
	\end{split}
\end{equation*}
We prove $\left\Vert{\mathcal I}\right\Vert_{\mathcal C_{\eta_{cu}-\nu}}=o\left(\left\Vert\xi-\xi_0\right\Vert\right)$ as $\xi\to\xi_0$, then we have 
\begin{equation*}
\begin{split}
\overline{v}\left(\cdot,\omega,\xi\right)-\overline{v}\left(\cdot,\omega,\xi_0\right)&-\mathcal{H}\left(\overline{v}\left(\cdot,\omega,\xi\right)-\overline{v}\left(\cdot,\omega,\xi_0\right)\right)\\&=\phi_{A_0}\left(t,0\right)\left(\xi-\xi_0\right)+\mathcal I\\&=	\phi_{A_0}\left(t,0\right)\left(\xi-\xi_0\right)+o\left(\left\Vert\xi-\xi_0\right\Vert\right),
\end{split}
\end{equation*}
which implies
\begin{equation*}
	\overline{v}\left(\cdot,\omega,\xi\right)-\overline{v}\left(\cdot,\omega,\xi_0\right)=\left(id-\mathcal{H}\right)^{-1}\phi_{A_0}\left(t,0\right)\left(\xi-\xi_0\right)+o\left(\left\Vert\xi-\xi_0\right\Vert\right).
\end{equation*}
We know that $\phi_{A_0}\left(t,0\right)=T\left(t\right)e^{\int_{0}^{t}z\left(\theta_r\omega\right)dr}$ is a bounded operator, therefore $\overline{v}\left(\cdot,\omega,\xi\right)$ is differentiable in $\xi$ and $D_\xi\overline{v}\left(\cdot,\omega,\xi\right)\in{\mathcal L}\left(X_{0c},\mathcal C_{\eta_{cu}+\nu}\right)$. Next we prove $\left\Vert{\mathcal I}\right\Vert_{\mathcal C_{\eta_{cu}+\nu}}=o\left(\left\Vert\xi-\xi_0\right\Vert\right)$ as $\xi\to\xi_0$. We write $e^{-\left(\eta_{cu}+\nu\right)t-\int_0^tz\left(\theta_r\omega\right)dr}\mathcal I$ as a sum of six terms, i.e., $e^{-\left(\eta_{cu}+\nu\right)t-\int_0^tz\left(\theta_r\omega\right)dr}\mathcal I=\sum_{i=1}^6\mathcal I_i$, where for $t<N_1$,
\begin{equation*}
\begin{split}
	\mathcal I_1=e^{-\left(\eta_{cu}+\nu\right)t-\int_{0}^tz\left(\theta_r\omega\right)dr}\int_{N_1}^t\phi_{A_{0c}}\left(t,s\right)\Pi_c\overline\Delta_G\left(s\right)ds,
	\end{split}
\end{equation*}
and $\mathcal I_1=0$ for $t\geq{N_1}$; For $t<N_1$,
\begin{equation*}
\begin{split}
	&\mathcal I_2=e^{-\left(\eta_{cu}+\nu\right)t-\int_{0}^tz\left(\theta_r\omega\right)dr}\int_{0}^{N_1}\phi_{A_{0c}}\left(t,s\right)\Pi_c\overline\Delta_G\left(s\right)ds,
	\end{split}
\end{equation*}
and change $N_1$ to $t$ if $t\geq{N_1}$. For $t<N_2$,
\begin{equation*}
\begin{split}
	\mathcal I_3=e^{-\left(\eta_{cu}+\nu\right)t-\int_{0}^tz\left(\theta_r\omega\right)dr}\int_{N_2}^{t}\phi_{A_{0u}}\left(t,s\right)\Pi_u\overline\Delta_G\left(s\right)ds,
	\end{split}
\end{equation*}
and $\mathcal I_3=0$ for $t\geq{N_2}$; For $t<N_2$,
\begin{equation*}
\begin{split}
	\mathcal I_4=e^{-\left(\eta_{cu}+\nu\right)t-\int_{0}^tz\left(\theta_r\omega\right)dr}\int_{0}^{N_2}\phi_{A_{0u}}\left(t,s\right)\Pi_u\overline\Delta_G\left(s\right)ds,
	\end{split}
\end{equation*}
and change $N_2$ to $t$ if $t\geq{N_2}$. For $t>N_3$,
\begin{equation*}
\begin{split}
	\mathcal I_5=e^{-\left(\eta_{cu}+\nu\right)t-\int_{0}^tz\left(\theta_r\omega\right)dr}\lim_{\lambda\to+\infty}\int_{N_3}^t\phi_{A_{0s}}\left(t,s\right)\lambda\left(\lambda{I}-A_s\right)^{-1}\Pi_s\overline\Delta_G\left(s\right)ds,
	\end{split}
\end{equation*}
and $\mathcal I_5=0$ for $t\leq{N_3}$; For $t>N_3$,
\begin{equation*}
\begin{split}
	\mathcal I_6=e^{-\left(\eta_{cu}+\nu\right)t-\int_{0}^tz\left(\theta_r\omega\right)dr}\lim_{\lambda\to+\infty}\int_{-\infty}^{N_3}\phi_{A_{0s}}\left(t,s\right)\lambda\left(\lambda{I}-A_s\right)^{-1}\Pi_s\overline\Delta_G\left)s\right)ds,
	\end{split}
\end{equation*}
and change $N_3$ to $t$ if $t\leq{N_3}$. $N_1$, $N_2$, $N_3$ above are negative numbers to be chosen later. By \eqref{eq4} and \eqref{eq17}, for $t<N_1$,
\begin{equation*}
	\begin{split}
	\left\vert{\mathcal I_1}\right\vert&\leq2KL\int_{t}^{N_1}e^{\gamma\left\vert{t-s}\right\vert}e^{-\left(\eta_{cu}+\nu\right)t}e^{\left(\eta_{cu}+2\nu\right)s}ds\left\Vert\overline{v}\left(\cdot,\omega,\xi\right)-\overline{v}\left(\cdot,\omega,\xi_0\right)\right\Vert_{\mathcal C_{\eta_{cu}+2\nu}}\\&\leq2KL\left\Vert\overline{v}\left(\cdot,\omega,\xi\right)-\overline{v}\left(\cdot,\omega,\xi_0\right)\right\Vert_{\mathcal C_{\eta_{cu}+2\nu}}e^{-\left(\gamma+\eta_{cu}+\nu\right)t}	\int_{t}^{N_1}e^{\left(\gamma+\eta_{cu}+2\nu\right)s}ds\\&\leq	\frac{-2K^2L}{\left(\gamma+\eta_{cu}+2\nu\right)\left[1-KL\left(C_{\zeta-2\nu}-\frac1{\gamma+\eta_{cu}+2\nu}-\frac1{\eta_{cu}-\alpha+2\nu}\right)\right]}e^{\nu{N_1}}\left\Vert\xi-\xi_0\right\Vert.
	\end{split}
\end{equation*} 
For given $\varepsilon>0$, choose $N_1$ so negative that 
\begin{equation}\label{eq21}
	\underset{t\leq0}{\sup}\left\vert{\mathcal I_1}\right\vert\leq\frac{\varepsilon}{6}\left\Vert\xi-\xi_0\right\Vert.
\end{equation}
Fix such $N_1$, for $t\leq0$,
\begin{equation*}
\begin{split}
	&\left\vert{\mathcal I_2}\right\vert\leq{K}\left\Vert\overline{v}\left(\cdot,\omega,\xi\right)-\overline{v}\left(\cdot,\omega,\xi_0\right)\right\Vert_{\mathcal C_{\eta_{cu}+\nu}}e^{-\left(\gamma+\eta_{cu}+\nu\right)t}	\int_{N_1}^{0}e^{\left(\gamma+\eta_{cu}+\nu\right)s}\\&\times\int_0^1\Vert{D}_vG\left(\theta_s\omega,\tau\overline{v}\left(s,\omega,\xi\right)+\left(1-\tau\right)\overline{v}\left(s,\omega,\xi_0\right)\right)-D_vG\left(\theta_s\omega,\overline{v}\left(s,\omega,\xi_0\right)\right)\Vert{d\tau}ds.
\end{split}
\end{equation*}
Since $G\left(\omega,v\right)$ is $C^k$, $D_vG\left(\omega,v\right)$ is continuous. For $\varepsilon>0$, there exists $\rho_1>0$ such that if $\left\Vert\xi-\xi_0\right\Vert<\rho_1$,
\begin{equation*}
\begin{split}
	&\left\Vert{D}_vG\left(\theta_s\omega,\tau\overline{v}\left(s,\omega,\xi\right)+\left(1-\tau\right)\overline{v}\left(s,\omega,\xi_0\right)\right)-D_vG\left(\theta_s\omega,\overline{v}\left(s,\omega,\xi_0\right)\right)\right\Vert\\&\leq\frac{-\varepsilon\left(\gamma+\eta_{cu}+\nu\right)\left[1-KL\left(C_{\zeta-\nu}-\frac1{\gamma+\eta_{cu}+\nu}-\frac1{\eta_{cu}-\alpha+\nu}\right)\right]}{6K^2e^{\left(\gamma+\eta_{cu}+\nu\right)N_1}}.
\end{split}
\end{equation*}
Then by \eqref{eq17},
\begin{equation}\label{eq22}
	\underset{t\leq0}{\sup}\left\vert{\mathcal I_2}\right\vert\leq\frac{\varepsilon}{6}\left\Vert\xi-\xi_0\right\Vert ~\text{for} ~\left\Vert\xi-\xi_0\right\Vert<\rho_1.
\end{equation}
Similarly, by choosing $N_2$ to be sufficiently negative, we have
\begin{equation}\label{eq23}
	\underset{t\leq0}{\sup}\left\vert{\mathcal I_3}\right\vert\leq\frac{\varepsilon}{6}\left\vert\xi-\xi_0\right\vert.
\end{equation}
Fix such $N_2$, there exists $\rho_2>0$ such that if $\left\Vert\xi-\xi_0\right\Vert<\rho_2$, 
\begin{equation}\label{eq24}
	\underset{t\leq0}{\sup}\left\vert{\mathcal I_4}\right\vert\leq\frac{\varepsilon}{6}\left\Vert\xi-\xi_0\right\Vert.
\end{equation} 	
And for $\mathcal I_6$, for $t>N_3$,
\begin{equation*}
	\begin{split}
	&	\left\vert{\mathcal I_6}\right\vert\\&=e^{-\left(\eta_{cu}+\nu\right)t-\int_{0}^tz\left(\theta_r\omega\right)dr}\bigg\vert\lim_{\lambda\to+\infty}\lim_{r\to-\infty}\int_{0}^{N_3-r}\phi_{A_{0s}}\left(t,l+r\right)\lambda{\left(\lambda I-A_s\right)^{-1}}\Pi_s\overline\Delta_G\left(l+r\right)dl\bigg\vert\\&\leq2LC_{\zeta-\nu}\underset{h\in(-\infty,N_3]}{\sup}e^{\left(\zeta-\nu\right)\left(t-h\right)}e^{-\left(\eta_{cu}+\nu\right)t}e^{\left(\eta_{cu}+2\nu\right)h}\left\Vert\overline{v}\left(h,\omega,\xi\right)-\overline{v}\left(h,\omega,\xi_0\right)\right\Vert_{\mathcal C_{\eta_{cu}+2\nu}}\\&\leq2LC_{\zeta-\nu}\underset{h\in(-\infty,N_3]}{\sup}e^{\left(\zeta-2\nu-\eta_{cu}\right)\left(t-h\right)}e^{\nu h}\left\Vert\overline{v}\left(\cdot,\omega,\xi\right)-\overline{v}\left(\cdot,\omega,\xi_0\right)\right\Vert_{\mathcal C_{\eta_{cu}+2\nu}}\\&\leq2LC_{\zeta-\nu}e^{\nu N_3}\left\Vert\overline{v}\left(\cdot,\omega,\xi\right)-\overline{v}\left(\cdot,\omega,\xi_0\right)\right\Vert_{\mathcal C_{\eta_{cu}+2\nu}}\\&\leq\frac{2KLC_{\zeta-\nu}e^{\nu{N_3}}}{1-KL\left(C_{\zeta-2\nu}-\frac1{\gamma+\eta_{cu+2\nu}}-\frac1{\eta_{cu}-\alpha+2\nu}\right)}\left\Vert\xi-\xi_0\right\Vert.
	\end{split}
\end{equation*}
Choose $N_3$ so negative that 
\begin{equation}\label{eq25}
	\underset{t\leq0}{\sup}\left\vert{\mathcal I_6}\right\vert\leq\frac{\varepsilon}{6}\left\Vert\xi-\xi_0\right\Vert.
\end{equation}
Fix such $N_3$, for $\mathcal I_5$, $t>N_3$, 
\begin{equation*}
	\begin{split}
		&\left\vert{\mathcal I_5}\right\vert=e^{-\left(\eta_{cu}+\nu\right)t-\int_{0}^tz\left(\theta_r\omega\right)dr}\bigg\vert\lim_{\lambda\to+\infty}\int_{0}^{t-N_3}\phi_{A_{0s}}\left(t,l+N_3\right)\lambda{\left(\lambda I-A_s\right)^{-1}}\Pi_s\overline\Delta_G\left(l+N_3\right)dl\bigg\vert\\&\leq C_{\zeta-2\nu}\underset{h\in[N_3,t]}{\sup}e^{\left(\zeta-2\nu\right)\left(t-h\right)}e^{-\left(\eta_{cu}+\nu\right)t}e^{\left(\eta_{cu}+2\nu\right)h}\left\Vert\overline{v}\left(h,\omega,\xi\right)-\overline{v}\left(h,\omega,\xi_0\right)\right\Vert_{\mathcal C_{\eta_{cu}+2\nu}}\\&\times\int_0^1\Vert{D}_vG\left(\theta_h\omega,\tau\overline{v}\left(h,\omega,\xi\right)+\left(1-\tau\right)\overline{v}\left(h,\omega,\xi_0\right)\right)-D_vG\left(\theta_h\omega,\overline{v}\left(h,\omega,\xi_0\right)\right)\Vert{d\tau}\\&\leq C_{\zeta-2\nu}\underset{h\in[N_3,t]}{\sup}e^{\left(\zeta-4\nu-\eta_{cu}\right)\left(t-h\right)}e^{\nu{t}}\left\Vert\overline{v}\left(\cdot,\omega,\xi\right)-\overline{v}\left(\cdot,\omega,\xi_0\right)\right\Vert_{\mathcal C_{\eta_{cu}+2\nu}}\\&\times\int_0^1\Vert{D}_vG\left(\theta_h\omega,\tau\overline{v}\left(h,\omega,\xi\right)+\left(1-\tau\right)\overline{v}\left(h,\omega,\xi_0\right)\right)-D_vG\left(\theta_h\omega,\overline{v}\left(h,\omega,\xi_0\right)\right)\Vert{d\tau}.
	\end{split}
\end{equation*}
Similarly, for $\varepsilon>0$, there exists $\rho_3>0$ such that if $\Vert\xi-\xi_0\Vert<\rho_3$,
\begin{equation*}
\begin{split}
	&\left\Vert{D}_vG\left(\theta_h\omega,\tau\overline{v}\left(h,\omega,\xi\right)+\left(1-\tau\right)\overline{v}\left(h,\omega,\xi_0\right)\right)-D_vG\left(\theta_h\omega,\overline{v}\left(h,\omega,\xi_0\right)\right)\right\Vert\\&\leq\frac{1-KL\left(C_{\zeta-\nu}-\frac1{\gamma+\eta_{cu}+\nu}-\frac1{\eta_{cu}-\alpha+\nu}\right)}{KC_{\zeta-2\nu}}.
\end{split}
\end{equation*}
Then by \eqref{eq17},
\begin{equation}\label{eq26}
	\underset{t\leq0}{\sup}\left\vert{\mathcal I_5}\right\vert\leq\frac{\varepsilon}{6}\left\Vert\xi-\xi_0\right\Vert ~\text{for} ~\left\Vert\xi-\xi_0\right\Vert<\rho_3.
\end{equation}
By taking $\rho_0=\min\left\{\rho_1,\rho_2,\rho_3\right\}$ and combining \eqref{eq21}-\eqref{eq26}, we have
\begin{equation*}
	\left\Vert{\mathcal I}\right\Vert_{\mathcal C_{\eta_{cu}-\nu}}\leq\varepsilon\left\Vert\xi-\xi_0\right\Vert~\text{for}~\left\Vert\xi-\xi_0\right\Vert<\rho_0. 
\end{equation*}
This implies $\left\Vert{\mathcal I}\right\Vert_{\mathcal C_{\eta_{cu}-\nu}}=o\left(\left\Vert\xi-\xi_0\right\Vert\right)$ as $\xi\to\xi_0$. Next we prove $D_\xi\overline{v}\left(\cdot,\omega,\xi\right)$ is continuous with respect to $\xi$ from $X_{0cu}$ to 	$\mathcal L\left(X_{0cu},\mathcal C_{\eta_{cu}}\right)$. By \eqref{eq16}, we have
\begin{equation*}
\begin{split}
	&D_\xi\overline{v}\left(t,\omega,\xi\right)=\phi_{A_{0c}}\left(t,0\right)\Pi_{0cu}\\&+\int_0^t\phi_{A_{0c}}\left(t,s\right)\Pi_cD_vG\left(\theta_s\omega,\overline{v}\left(s,\omega,\xi\right )\right)D_\xi\overline{v}\left(s,\omega,\xi\right)ds\\&+\int_0^{t}\phi_{A_{0u}}\left(t,s\right)\Pi_uD_vG\left(\theta_s\omega,\overline{v}\left(s,\omega,\xi\right)\right)D_\xi\overline{v}\left(s,\omega,\xi\right)ds\\&+\lim_{\lambda\to+\infty}\int_{-\infty}^t\phi_{A_{0s}}\left(t,s\right)\lambda\left(\lambda I-A_s\right)^{-1}\Pi_sD_vG\left(\theta_s\omega,\overline{v}\left(s,\omega,\xi\right)\right)D_\xi\overline{v}\left(s,\omega,\xi\right)ds. 
\end{split}
\end{equation*}
For $\xi\in{X_{0cu}}$, $v\in{\mathcal C_{\eta_{cu}}}$, define the operator $\mathcal{H}^{'}:\mathcal C_{\eta_{cu}}\to{\mathcal C_{\eta_{cu}}}$ as follows
\begin{equation*}
\begin{split}
	\mathcal{H^{'}}v=&\int_0^t\phi_{A_{0c}}\left(t,s\right)\Pi_cD_vG\left(\theta_s\omega,\overline{v}\left(s,\omega,\xi\right)\right)v\left(s\right)ds\\&+\int_0^{t}\phi_{A_{0u}}\left(t,s\right)\Pi_uD_vG\left(\theta_s\omega,\overline{v}\left(s,\omega,\xi\right)\right)v\left(s\right)ds\\&+\lim_{\lambda\to+\infty}\int_{-\infty}^t\phi_{A_{0s}}\left(t,s\right)\lambda\left(\lambda I-A_s\right)^{-1}\Pi_sD_vG\left(\theta_s\omega,\overline{v}\left(s,\omega,\xi\right)\right)v\left(s\right)ds.
\end{split}
\end{equation*}
Similarly, we have
\begin{equation}\label{eq27}
	\left\Vert{\mathcal{H}'}\right\Vert\leq{K}L\left(C_{\zeta-\nu}-\frac1{\gamma+\eta_{cu}+\nu}-\frac1{\eta_{cu}-\alpha+\nu}\right)<1.
\end{equation}
For $\xi,\xi_0\in{X_{0cu}}$, 
\begin{equation}\label{eq28}
	\begin{split}
		&D_\xi\overline{v}\left(t,\omega,\xi\right)-D_\xi\overline{v}\left(t,\omega,\xi_0\right)\\&=\int_0^t\phi_{A_{0c}}\left(t,s\right)\Pi_c\Big[D_vG\left(\theta_s\omega,\overline{v}\left(s,\omega,\xi\right)\right)D_\xi\overline{v}\left(s,\omega,\xi\right)\\&-D_vG\left(\theta_s\omega,\overline{v}\left(s,\omega,\xi_0\right)\right)D_\xi\overline{v}\left(s,\omega,\xi_0\right)\Big]ds\\&+\int_{0}^{t}\phi_{A_{0u}}\left(t,s\right)\Pi_u\Big[D_vG\left(\theta_s\omega,\overline{v}\left(s,\omega,\xi\right)\right)D_\xi\overline{v}\left(s,\omega,\xi\right)\\&-D_vG\left(\theta_s\omega,\overline{v}\left(s,\omega,\xi_0\right)\right)D_\xi\overline{v}\left(s,\omega,\xi_0\right)\Big]ds\\&+\lim_{\lambda\to+\infty}\int_{-\infty}^t\phi_{A_{0s}}\left(t,s\right)\lambda\left(\lambda I-A_s\right)^{-1}\Pi_s\Big[D_vG\left(\theta_s\omega,\overline{v}\left(s,\omega,\xi\right)\right)D_\xi\overline{v}\left(s,\omega,\xi\right)\\&-D_vG\left(\theta_s\omega,\overline{v}\left(s,\omega,\xi_0\right)\right)D_\xi\overline{v}\left(s,\omega,\xi_0\right)\Big]ds\\&=\mathcal{H}^{'}\left(D_\xi\overline{v}\left(t,\omega,\xi\right)-D_\xi\overline{v}\left(t,\omega,\xi_0\right)\right)+{\mathcal I}',
	\end{split}
\end{equation}
where
\begin{equation*}
\begin{split}
{\mathcal I}^{'}&=\int_0^t\phi_{A_{0c}}\left(t,s\right)\Pi_c\Big[D_vG\left(\theta_s\omega,\overline{v}\left(s,\omega,\xi\right)\right)-D_vG\left(\theta_s\omega,\overline{v}\left(s,\omega,\xi_0\right)\right)\Big]D_\xi\overline{v}\left(s,\omega,\xi_0\right)ds\\&+\int_{0}^{t}\phi_{A_{0u}}\left(t,s\right)\Pi_u\Big[D_vG\left(\theta_s\omega,\overline{v}\left(s,\omega,\xi\right)\right)-D_vG\left(\theta_s\omega,\overline{v}\left(s,\omega,\xi_0\right)\right)\Big]D_\xi\overline{v}\left(s,\omega,\xi_0\right)ds\\&+\lim_{\lambda\to+\infty}\int_{-\infty}^t\phi_{A_{0s}}\left(t,s\right)\lambda\left(\lambda I-A_s\right)^{-1}\Pi_s\Big[D_vG\left(\theta_s\omega,\overline{v}\left(s,\omega,\xi\right)\right)\\&-D_vG\left(\theta_s\omega,\overline{v}\left(s,\omega,\xi_0\right)\right)\Big]D_\xi\overline{v}\left(s,\omega,\xi_0\right)ds.
\end{split}
\end{equation*}
By \eqref{eq27}, $id-\mathcal{H}^{'}$ has a bounded inverse in $ \mathcal L\left(X_{0cu},\mathcal C_{\eta_{cu}}\right)$. From \eqref{eq28}, we have
\begin{equation*}
	D_\xi\overline{v}\left(t,\omega,\xi\right)-D_\xi\overline{v}\left(t,\omega,\xi_0\right)=\left(id-\mathcal{H}^{'}\right)^{-1}{\mathcal I}^{'}.
\end{equation*}
Then by \eqref{eq28},
\begin{equation}\label{eq29}
	\Vert{D_\xi\overline{v}\left(t,\omega,\xi\right)-D_\xi\overline{v}\left(t,\omega,\xi_0\right)}\Vert_{ \mathcal L\left(X_{0cu},\mathcal C_{\eta_{cu}}\right)}\leq\frac{{\left\Vert{{\mathcal I}^{'}}\right\Vert}_{\mathcal L\left(X_{0cu},\mathcal C_{\eta_{cu}}\right)}}{1-{K}L\left(C_{\zeta}-\frac1{\gamma+\eta_{cu}}-\frac1{\eta_{cu}-\alpha}\right)}.
\end{equation}
Using the same procedure as for $\mathcal I$, we have ${\Vert{{\mathcal I}^{'}}\Vert}_{ \mathcal L\left(X_{0cu},\mathcal C_{\eta_{cu}}\right)}=o\left(1\right)$ as $\xi\to\xi_0$. Then by \eqref{eq29}, $D_\xi\overline{v}\left(\cdot,\omega,\xi\right)$ is continuous with respect to $\xi$. So we proved that $\overline{v}\left(t,\omega,\xi\right)$ is $C^1$. 
	
\noindent {\bf Step 2.} We prove $\overline{v}\left(t,\omega,\xi\right)$ is $C^k$. 

Make the inductive assumption that $\overline{v}\left(\cdot,\omega,\xi\right)$ is $C^j$ from $X_{0cu}$ to $\mathcal C_{j\eta_{cu}}$ for all $1\leq{j}\leq{m-1}$ and we prove it for $j=m$. By calculation, the $\left(m-1\right)$th-derivative $D^{m-1}_\xi\overline{v}\left(\cdot,\omega,\xi\right)$ satiefies the following equation
\begin{equation*}
	\begin{split}
		&D^{m-1}_\xi\overline{v}\left(t,\omega,\xi\right)=\int_0^t\phi_{A_{0c}}\left(t,s\right)\Pi_cD_vG\left(\theta_s\omega,\overline{v}\left(s,\omega,\xi\right)\right)D^{m-1}_\xi\overline{v}\left(s,\omega,\xi\right)ds\\&+\int_0^{t}\phi_{A_{0u}}\left(t,s\right)\Pi_uD_vG\left(\theta_s\omega,\overline{v}\left(s,\omega,\xi\right)\right)D^{m-1}_\xi\overline{v}\left(s,\omega,\xi\right)ds\\&+\lim_{\lambda\to+\infty}\int_{-\infty}^t\phi_{A_{0s}}\left(t,s\right)\lambda\left(\lambda I-A_s\right)^{-1}\Pi_sD_vG\left(\theta_s\omega,\overline{v}\left(s,\omega,\xi\right)\right)D^{m-1}_\xi\overline{v}\left(s,\omega,\xi\right)ds\\&+\int_0^t\phi_{A_{0c}}\left(t,s\right)\Pi_cR_{m-1}\left(s,\omega,\xi\right)ds+\int_0^{t}\phi_{A_{0u}}\left(t,s\right)\Pi_uR_{m-1}\left(s,\omega,\xi\right)ds\\&+\lim_{\lambda\to+\infty}\int_{-\infty}^t\phi_{A_{0s}}\left(t,s\right)\lambda\left(\lambda I-A_s\right)^{-1}\Pi_sR_{m-1}\left(s,\omega,\xi\right)ds,
	\end{split}
\end{equation*}
where
\begin{equation*}
	R_{m-1}\left(s,\omega,\xi\right)=\sum_{i=1}^{m-3}\binom{m-2}{i}D^{m-2-i}_\xi\left(D_vG\left(\theta_s\omega,\overline{v}\left(s,\omega,\xi\right)\right)\right)D^{i+1}_\xi\overline{v}\left(s,\omega,\xi\right).
\end{equation*}
Applying the chain rule to $D^{m-2-i}_\xi\left(D_vG\left(\theta_s\omega,\overline{v}\left(s,\omega,\xi\right)\right)\right)$, we observe that each term in $R_{m-1}\left(s,\omega,\xi\right)$ contains factors: $D^{i_1}_v{G\left(\theta_s\omega,\overline{v}\left(s,\omega,\xi\right)\right)}$ for some $2\leq{i_1}\leq{m-1}$ and at least two derivatives $D^{i_2}_\xi\overline{v}\left(s,\omega,\xi\right)$ and $D^{i_3}_\xi\overline{v}\left(s,\omega,\xi\right)$ for some $i_2,i_3\in\{1,...,m-2\}$. Since $D^{i}_\xi\overline{v}\left(\cdot,\omega,\xi\right)\in{\mathcal C_{i\eta_{cu}}}$ for $i=1,...,m-1$ and $F$ is $C^k$, $R_{m-1}\left(\cdot,\omega,\xi\right):X_{0cu}\to{L^{m-1}\left(X_{0cu},\mathcal C_{\left(m-1\right)\eta_{cu}+\nu}\right)}$ are $C^1$ in $\xi$. Using the same argument we used in the case $k=1$, we can show that $D^{m-1}_\xi\overline{v}\left(\cdot,\omega,\xi\right)$ is $C^1$ from $X_{0cu}$ to $L^m\left(X_{0cu},\mathcal C_{m\eta_{cu}}\right)$. So $\overline{v}\left(t,\omega,\xi\right)$ is $C^k$.
Thus $h^{cu}\left(\xi,\omega\right)$ is $C^k$, which gurantees the $C^k$-smoothness of $\mathcal{M}^{cu}\left(\omega\right)$ and  $\mathcal{M}^{cu*}\left(\omega\right)$.

\end{proof}
The existence and smoothness of the center-stable manifold $\mathcal M^{cs*}\left(\omega\right)$ of \eqref{eq1} can be obtained through similar arguments.

\section{Smooth Center-stable foliations}\label{sec4}

In this section, we study the existence of invariant foliations of \eqref{eq1}. We first introduce the concept of invariant foliations, which relies on the decomposition of the space $X$ as well. Fix $\omega\in\Omega$, let $\{\mathcal W\left(x,\omega\right):x\in X_0\}$ be a family of submanifolds of $X_0$ parametrized by $x\in X_0$. 
\begin{definition}\label{def4.1}
	A family of submanifolds $\{\mathcal W^{cs}\left(x,\omega\right):x\in X_0\}$ is called a $C^{k}$ center-stable foliation for $X_0$ if the following conditions are satisfied:
	\begin{itemize}
		\item [$\left(\mathrm{i}\right)$]$x\in\mathcal W^{cs}\left(x,\omega\right)$ for each $x\in X_0$.
		\item [$\left(\mathrm{ii}\right)$]$\mathcal W^{cs}\left(x,\omega\right)$ and $\mathcal W^{cs}\left(\tilde x,\omega\right)$ are either disjoint or identical for each $x,\tilde x\in X_0$.
		\item [$\left(\mathrm{iii}\right)$]$\mathcal W^{cs}\left(0,\omega\right)$ is tangent to $X_{0cs}$ at the origin, every leaf $\mathcal W^{cs}\left(x,\omega\right)$ is given by the graph of a $C^k$ function $l^{cs}\left(\cdot,\omega,x\right):X_{0cs}\to X_{0u}$, i.e.,
		\begin{equation*}
			\mathcal W^{cs}=\{\iota+l^{cs}\left(\iota,\omega,x\right):\iota\in X_{0cs}\}.
		\end{equation*}
		
	\end{itemize}
	Accordingly, a $C^k$ center-unstable(resp. stable, unstable) foliation $\mathcal W^{cu}\left(\omega\right)$(resp. $\mathcal W^{s}\left(\omega\right)$, $\mathcal W^{u}\left(\omega\right)$) is given by the graph of a Lipschitz map $h^i\left(\cdot,\omega\right):X_{0i}\to X_{0j}$, $i=cu$(resp. $s, u$) and $j=s$(resp. $j=cu, cs$).
\end{definition}

Define
\begin{equation*}
	\mathcal{W}^{cs}\left(x,\omega\right)=\{\tilde{x}\in X_0:v\left(t,\omega,\tilde{x}\right)-v\left(t,\omega,x\right)\in\mathcal{C}_{\eta_{cs}}\left([0,+\infty);X_0\right)\}.
\end{equation*}
where $v\left(t,\omega,x\right)$ is the solution of \eqref{eq4} with initial data $v\left(0\right)=x$. We below prove that $\mathcal{W}^{cs}\left(x,\omega\right)$ is invariant and it is given by the graph of a Lipschitz function.  
\begin{theorem}\label{thm4.2}
	Assume $-\beta<\chi<\eta_{cs}$ and $\gamma<\eta_{cs}<\min\{\alpha,\beta\}$ with 
\begin{equation*}
	{KL\left(C_\chi+\frac1{\eta_{cs}-\gamma}+\frac1{\alpha-\eta_{cs}}\right)}<1,
\end{equation*}
and assume for $\sigma^*>0$ such that for $0\leq\sigma\leq\sigma^*$, $\gamma<\eta_{cs}-2\sigma<\eta_{cs}-\sigma<\min\{\alpha,\beta\}$, $-\beta<\chi<\eta_{cs}-2\sigma<\eta_{cs}-\sigma$ with
\begin{equation}\label{eq30}
	KL\left(C_{\chi+i\sigma}+\frac1{\eta_{cs}-i\sigma-\gamma}+\frac1{\alpha-i\eta_{cs}+i\sigma}\right)<\frac16,~i=1,2,
\end{equation}
There exists a Lipschitz invariant center-stable foliation for \eqref{eq4} whose leaf is given by 
\begin{equation*}
	\mathcal{W}^{cs}\left(x,\omega\right)=\{\iota+l^{cs}\left(\iota,\omega,x\right):\iota\in X_{0cs}\},
\end{equation*}
where $x\in X_0$ and $l^{cs}\left(\cdot,\omega,x\right):X_{0cs}\to X_{0u}$ is measurable with respect to $\left(\iota,\omega,x\right)$ and Lipschitz continuous in $\iota$ with
\begin{equation*}
	Lip~l^{cs}\left(\cdot,\omega,x\right)\leq K_s=\frac{K^2L}{\left(\alpha-\eta_{cs}\right)1-KL\left(C_\chi+\frac1{\eta_{cs }-\gamma}+\frac1{\alpha-\eta_{cs}}\right)}\end{equation*}
Moreover, if $K_uK_s<1$($K_u$is defined in Theorem \ref{thm3.4}), then $\mathcal W^{cs}\left(x,\omega\right)\cap\mathcal M^{cu}\left(\omega\right)$ contains a unique point.
\end{theorem}
\begin{proof}We proceed it in five steps.

\noindent {\bf Step 1.} 
We prove that $\tilde{x}\in{\mathcal{W}^{cs}\left(x,\omega\right)}$ if and only if there exists a function $\psi\left(\cdot\right)\in{\mathcal{C}_{\eta_{cs}}\left([0,+\infty);X_0\right)}$ with $\psi\left(0\right)=\tilde{x}-x$ which satisfies
\begin{equation}\label{eq31}
\begin{split}
\psi\left(t\right)=&\phi_{A_{0c}}\left(t,0\right)\iota+\Pi_{0s}\left(S\diamond \Delta G_x\left(v,\psi\right)\right)\left(t\right)\\&+\int_0^t\phi_{A_{0c}}\left(t,s\right)\Pi_c \Delta G\left(\theta_s\omega,v\left(s,\omega,x\right),\psi\left(s\right)\right)ds\\&+\int_{+\infty}^{t}\phi_{A_{0u}}\left(t,s\right)\Pi_u \Delta G\left(\theta_s\omega,v\left(s,\omega,x\right),\psi\left(s\right)\right)ds
\end{split}	
\end{equation}
where $\iota=\Pi_{0cs}\left({\tilde{x}-x}\right)$ and 
\begin{equation*}	
	\Delta G\left(\theta_s\omega,v\left(s,\omega,x\right),\psi\left(s\right)\right)=G\left(\theta_s\omega,v\left(s,\omega,x\right)+\psi\left(s\right)\right)-G\left(\theta_s\omega,v\left(s,\omega,x\right)\right) ,
\end{equation*}
\begin{equation*}
	\left(S\diamond \Delta G_x\left(v,\psi\right)\right)\left(t\right)=\left(S\diamond G_x\left(v+\psi\right)\right)\left(t\right)-\left(S\diamond G_x\left(v\right)\right)\left(t\right),
\end{equation*}
To this aim, we let $\tilde{x}\in{\mathcal{W}^{cs}\left(x,\omega\right)}$ and assume that $v\in{\mathcal{C}_{\eta_{cs}}\left([0,+\infty);X_0\right)}$ is a solution of \eqref{eq4}. Let $\psi\left(t\right)=v\left(t,\omega,\tilde{x}\right)-v\left(t,\omega,x\right)$. By the variation of constants formula, one further gets 
\begin{equation*}\begin{split}
	\Pi_{0cs}v\left(t,\omega,x\right)&=\phi_{A_{0c}}\left(t,0\right)\Pi_{0cs}x+\Pi_{0s}\left(S\diamond{G_x\left(v\right)}\right)\left(t\right)\\&+\int_0^t\phi_{A_{0c}}\left(t,s\right)\Pi_c G\left(\theta_s\omega,v\left(s,\omega,x\right)\right)ds.
	\end{split}
\end{equation*}
 and for all $t,l\in\mathbb{R}$ with $t<l$, 
\begin{equation*}
\begin{split}
	\Pi_{0u}v\left(t\right)=\phi_{A_{0u}}\left(t,l\right)\Pi_{0u}v\left(l\right)+\int_l^t\phi_{A_{0s}}\left(t,s\right)\Pi_u G\left(\theta_s\omega,v\left(s\right)\right)ds.
\end{split}
\end{equation*}
Then we get that
\begin{equation}\label{eq32}
\begin{split}
	\Pi_{0cs}\psi\left(t\right)&=\phi_{A_{0c}}\left(t,0\right)\iota+\Pi_{0s}\left(S\diamond \Delta G_x\left(v,\psi\right)\right)\left(t\right)\\&+\int_0^t\phi_{A_{0c}}\left(t,s\right)\Pi_c \Delta G\left(\theta_s\omega,v\left(s,\omega,x\right),\psi\left(s\right)\right)ds.
	\end{split}
\end{equation}
\begin{equation*}
\begin{split}
	\Pi_{0u}\psi\left(t\right)&=\phi_{A_{0u}}\left(t,l\right)\Pi_{0u}\psi\left(l\right)+\int_l^{t}\phi_{A_{0s}}\left(t,s\right)\Pi_u \Delta G\left(\theta_s\omega,v\left(s,\omega,x\right),\psi\left(s\right)\right)ds.
	\end{split}
\end{equation*}
According to \eqref{eq7}, for $l>0$,
\begin{equation*}
	\begin{split}
	\left\Vert\phi_{A_{0u}}\left(t,l\right)\Pi_{0u}\psi\left(l\right)\right\Vert\leq&{K}e^{\alpha\left(t-l\right)+\int_l^tz\left(\theta_r\omega\right)dr}\left\Vert{\psi\left(l\right)}\right\Vert\\\leq&{K}e^{\alpha\left(t-l\right)-\eta_{cs}{l}+\int_0^t{z\left(\theta_r\omega\right)dr}}e^{\eta_{cs}{l}-\int_0^l{z\left(\theta_r\omega\right)dr}}\left\Vert{\psi\left(l\right)}\right\Vert\\\leq&{K}e^{\left(\eta_{cs}-\alpha\right)l-\alpha{t}+\int_0^t{z\left(\theta_r\omega\right)dr}}{\left\Vert{\psi}\right\Vert}_{\mathcal C_{\eta_{cs}}}.	
	\end{split}
\end{equation*}
Let $l\to+\infty$, since $\eta_{cs}-\alpha<0$, we have
\begin{equation}\label{eq33}
	\Pi_{0u}\psi\left(t\right)=\int_{+\infty}^{t}\phi_{A_{0u}}\left(t,s\right)\Pi_u \Delta G\left(\theta_s\omega,v\left(s,\omega,x\right),\psi\left(s\right)\right)ds.
\end{equation}
By summing up \eqref{eq32} and \eqref{eq32}, we get \eqref{eq31}. The converse can be obtained by a straight computation.

\noindent {\bf Step 2.} 
We claim that \eqref{eq32} has a unique solution in $\mathcal{C}_{\eta_{cs}}\left([0,+\infty);X_0\right)$ with initial condition $\Pi_{0cs}\psi\left(0,\omega,\iota,x\right)=\iota$ for all $\iota\in{X_{0cs}}$. 

To show this claim, we denote $\mathcal{Z}\left(\psi,\iota\right)$ the right side of \eqref{eq32}, i.e.
\begin{equation}\label{eq34}
\begin{split}
	\mathcal{Z}\left(\psi,\iota\right)=&\phi_{A_{0c}}\left(t,0\right)\iota+\Pi_{0s}\left(S\diamond \Delta G_x\left(v,\psi\right)\right)\left(t\right)\\&+\int_0^t\phi_{A_{0c}}\left(t,s\right)\Pi_c \Delta G\left(\theta_s\omega,v\left(s,\omega,x\right),\psi\left(s\right)\right)ds\\&+\int_{+\infty}^{t}\phi_{A_{0u}}\left(t,s\right)\Pi_u \Delta G\left(\theta_s\omega,v\left(s,\omega,x\right),\psi\left(s\right)\right)ds
\end{split}
\end{equation}
It needs to show that $\mathcal{Z}$ maps from $\mathcal{C}_{\eta_{cs}}\left([0,+\infty);X_0\right)\times{X_{0cs}}$ to $\mathcal{C}_{\eta_{cs}}\left([0,+\infty);X_0\right)$ and it is a uniform contraction. For $\psi,\tilde{\psi}\in{\mathcal{C}_{\eta_{cs}}\left([0,+\infty);X_0\right)}$, $-\beta<\chi<\eta_{cs}$,
\begin{equation*}
\begin{split}
	&{\left\Vert{\mathcal{Z}\left(\psi,\iota\right)-\mathcal{Z}\left(\tilde{\psi},\iota\right)}\right\Vert}_{\mathcal{C}_{\eta_{cs}}}=\underset{t\geq0}{\sup}~e^{-\eta_{cs}t-\int_0^tz\left(\theta_s\omega\right)ds}{\left\Vert{\mathcal{Z}\left(\psi,\iota\right)-\mathcal{Z}\left(\tilde{\psi},\iota\right)}\right\Vert}\\&\leq LC_\chi\sup_{s\in[0,t]}e^{\left(\chi-\eta_{cs}\right)\left(t-s\right)}\left\Vert \psi-\tilde \psi\right\Vert_{\mathcal C_{\eta_{cs}}}+KL\sup_{t\geq0}\int_0^te^{\left(\eta_{cs}-\gamma\right)\left(s-t\right)-\eta_{cs}s-\int_0^sz\left(\theta_r\omega\right)dr}\left\Vert \psi-\tilde \psi\right\Vert ds\\&+KL\sup_{t\geq0}\int_t^{+\infty}e^{\left(\eta_{cs}-\alpha\right)\left(s-t\right)-\eta_{cs}s-\int_0^sz\left(\theta_r\omega\right)dr}\left\Vert \psi-\tilde \psi\right\Vert ds\\&\leq KL\left(C_\chi+\frac{1}{\eta_{cs}-\gamma}+\frac{1}{\alpha-\eta_{cs}}\right)\left\Vert \psi-\tilde \psi\right\Vert_{\mathcal C_{\eta_{cs}}},
\end{split}
\end{equation*}
which implies that $\mathcal{Z}$ is a uniform contraction with respect to $\psi$. Furthermore, it is straightforward that $\mathcal{Z}\left(\cdot,\iota\right)$ is well defined from $\mathcal{C}_{\eta_{cs}}\left([0,+\infty);X_0\right)\times{X_{0cs}}$ to $\mathcal{C}_{\eta_{cs}}\left([0,+\infty);X_0\right)$ by setting $\tilde{\psi}=0$. From the  contraction mapping principle, for any given $\iota\in{X_{0cs}}$, $\mathcal{Z}\left(\cdot,\iota\right)$ has a unique fixed point $\overline{\psi}\in{\mathcal{C}_{\eta_{cs}}\left([0,+\infty);X_0\right)}$. That is, for each $t\geq0$, $\overline{\psi}\left(t,\omega,\iota,x\right)$ is a unique solution to \eqref{eq34} and 
\begin{equation*}
	\mathcal{Z}\left(\overline{\psi},\iota\right)=\overline{\psi}\left(t,\omega,\iota,x\right).
\end{equation*} 
Also, for $\forall{\iota_1,\iota_2\in{X_{0cs}}}$, 
\begin{equation*}	
\begin{split}
	&{\left\Vert{\mathcal{Z}\left(\overline{\psi},\iota_1\right)-\mathcal{Z}\left(\overline{\psi},\iota_2\right)}\right\Vert}_{\mathcal{C}_{\eta_{cs}}}\\&\quad={\left\Vert{\overline{\psi}\left(t,\omega,\iota_1,x\right)-\overline{\psi}\left(t,\omega,\iota_2,x\right)}\right\Vert}_{\mathcal{C}_{\eta_{cs}}}\\&\quad\leq{K}\left\Vert{\iota_1-\iota_2}\right\Vert+KL\left(C_\chi+\frac{1}{\eta_{cs}-\gamma}+\frac{1}{\alpha-\eta_{cs}}\right){\left\Vert{\overline{\psi}\left(t,\omega,\iota_1,x\right)-\overline{\psi}\left(t,\omega,\iota_2,x\right)}\right\Vert}_{\mathcal{C}_{\eta_{cs}}}.
\end{split}
\end{equation*}  
Thus
\begin{equation}\label{eq35}
	{\left\Vert{\overline{\psi}\left(t,\omega,\iota_1,x\right)-\overline{\psi}\left(t,\omega,\iota_2,x\right)}\right\Vert}_{\mathcal{C}_{\eta_{cs}}}\leq\frac{K}{1-KL\left(C_\chi+\frac1{\eta_{cs}-\gamma}+\frac1{\alpha-\eta_{cs}}\right)}\left\Vert{\iota_1-\iota_2}\right\Vert,
\end{equation}
So $\overline{\psi}\left(t,\omega,\cdot,x\right)$ is Lipschitz continuous from $X_{0cu}$ to $\mathcal{C}_{\eta_{cs}}\left([0,+\infty);X_0\right)$. 

\noindent {\bf Step 3.} We prove that $\overline\psi\left(\cdot,\omega,\iota,x\right)$ is continuous with respect to $\left(\iota,x\right)$. By \eqref{eq30}, there is a fixed point in $\mathcal C_{\eta_{cs}-\sigma}$ and \eqref{eq35} holds with $\eta_{cs}$ replaced by $\eta_{cs}-\sigma$. Then we have
\begin{equation*}
	{\left\Vert{\overline{\psi}\left(t,\omega,\iota_1,x\right)-\overline{\psi}\left(t,\omega,\iota_2,x\right)}\right\Vert}_{\mathcal{C}_{\eta_{cs}-\sigma}}\leq\frac65K\Vert\iota_1-\iota_2\Vert~\text{and}~\Vert\overline\psi\left(t,\omega,\iota_1,x\right)\Vert\leq\frac65K\Vert\iota_1\Vert,
\end{equation*}
For $\left(\iota_1,x_1\right)$, $\left(\iota_2,x_2\right)\in X_{0cs}\times X_0$, there exists $\overline\psi\left(\cdot,\omega,\iota_1,x_1\right),\overline\psi\left(\cdot,\omega,\iota_2,x_2\right)\in\mathcal C_{\eta_{cs}-\sigma}\subset\mathcal C_{\eta_{cs}}$. For $t\geq0$,
\begin{equation*}\begin{split}
	&\overline\psi\left(\cdot,\omega,\iota_1,x_1\right)-\overline\psi\left(\cdot,\omega,\iota_2,x_2\right)=\phi_{A_{0c}}\left(t,0\right)\left(\iota_1-\iota_2\right)+\int_0^t\phi_{A_{0c}}\left(t,s\right)\Pi_c \Delta \mathcal G\left(s,\omega\right)ds\\&+\int_{+\infty}^t\phi_{A_{0u}}\left(t,s\right)\Pi_u \Delta \mathcal G\left(s,\omega\right)ds+\lim_{\lambda\to+\infty}\int_{0}^{t}\phi_{A_{0s}}\left(t,s\right)\lambda\left(\lambda I-A_s\right)^{-1}\Pi_s \Delta \mathcal G\left(s,\omega\right)ds
	\end{split}
\end{equation*} 
where
\begin{equation*}
	\Delta \mathcal G\left(s,\omega\right)=\Delta G\left(\theta_s\omega,v\left(s,\omega,x_1\right),\overline\psi\left(s,\omega,\iota_1,x_1\right)\right)-\Delta G\left(\theta_s\omega,v\left(s,\omega,x_2\right),\overline\psi\left(s,\omega,\iota_2,x_2\right)\right).
\end{equation*}
We prove that $\Vert\overline\psi\left(\cdot,\omega,\iota_1,x_1\right)-\overline\psi\left(\cdot,\omega,\iota_2,x_2\right)\Vert_{\mathcal C_{\eta_{cs}}}=o\left(1\right)$ as $\left(\iota_1,x_1\right)\to\left(\iota_2,x_2\right)$. Set for $t> N_1$,
\begin{equation*}
	\mathcal E_1=e^{-\eta_{cs}t-\int_0^tz\left(\theta_r\omega\right)dr}\int_{N_1}^t\phi_{A_{0c}}\left(t,s\right)\Pi_c\Delta \mathcal G\left(s,\omega\right)ds,
\end{equation*}
and $\mathcal E_1=0$ for $t\leq N_1$; For $t> N_1$,
\begin{equation*}
	\mathcal E_2=e^{-\eta_{cs}t-\int_0^tz\left(\theta_r\omega\right)dr}\int_0^{N_1}\phi_{A_{0c}}\left(t,s\right)\Pi_c\Delta \mathcal G\left(s,\omega\right)ds,
\end{equation*}
and change $N_1$ to $t$ if $t\leq N_1$. For $t>N_2$,
\begin{equation*}
	\mathcal E_3=-e^{-\eta_{cs}t-\int_0^tz\left(\theta_r\omega\right)dr}\int_{N_2}^t\phi_{A_{0u}}\left(t,s\right)\Pi_c\Delta \mathcal G\left(s,\omega\right)ds,
\end{equation*}  
and $\mathcal E_3=0$ for $t\leq N_2$; For $t> N_2$,
\begin{equation*}
	\mathcal E_4=-e^{-\eta_{cs}t-\int_0^tz\left(\theta_r\omega\right)dr}\int_{-\infty}^{N_2}\phi_{A_{0u}}\left(t,s\right)\Pi_c\Delta \mathcal G\left(s,\omega\right)ds,
\end{equation*}
and change $N_2$ to $t$ if $t\leq N_2$. For $t>N_3$,
\begin{equation*}
	\mathcal E_5=e^{-\eta_{cs}t-\int_0^tz\left(\theta_r\omega\right)dr}\lim_{\lambda\to+\infty}\int_{N_3}^t\phi_{A_{0s}}\left(t,s\right)\lambda\left(\lambda I-A_s\right)^{-1}\Pi_s\Delta \mathcal G\left(s,\omega\right)ds,
\end{equation*}  
and $\mathcal E_5=0$ for $t\leq N_3$; For $t> N_3$,
\begin{equation*}
	\mathcal E_6=e^{-\eta_{cs}t-\int_0^tz\left(\theta_r\omega\right)dr}\lim_{\lambda\to+\infty}\int_{0}^{N_3}\phi_{A_{0s}}\left(t,s\right)\lambda\left(\lambda I-A_s\right)^{-1}\Pi_s\Delta \mathcal G\left(s,\omega\right)ds,
\end{equation*}
and change $N_3$ to $t$ if $t\leq N_3$. $N_1$, $N_2$, $N_3$ above are large positive numbers to be chosen later. For $t>N_1$, and choose $\Vert\iota_1-\iota_2\Vert\leq1$,
\begin{equation*}\begin{split}
	\vert\mathcal E_1\vert\leq& KL\int_{N_1}^te^{\left(\eta_{cs}-\sigma-\gamma\right)\left(s-t\right)}e^{-\sigma t}ds[\Vert \overline\psi\left(\cdot,\omega,\iota_1,x_1\right)\Vert_{\mathcal C_{\eta_{cs}-\sigma}}+\Vert\overline\psi\left(\cdot,\omega,\iota_2,x_2\right)\Vert_{\mathcal C_{\eta_{cs}-\sigma}}]\\&\leq \frac{2K^2Le^{-\sigma N_1}}{1-KL\left(C_{\chi+\sigma}+\frac1{\eta_{cs}-\sigma-\gamma}+\frac1{\alpha-\eta_{cs}+\sigma}\right)}\left(2\left\Vert{\iota_1}\right\Vert+1\right)\\&\leq\frac{12}{5}KLe^{-\sigma N_1}\left(2\left\Vert{\iota_1}\right\Vert+1\right).
\end{split}
\end{equation*}
For given $\varepsilon>0$, choose $N_1$ so positive that
\begin{equation}\label{eq36}
	\sup_{t\geq0}\vert\mathcal E_1\vert\leq\frac{\varepsilon}{14}.
\end{equation}
For such $N_1$, for $t>N_1$,
\begin{equation*}
	\begin{split}
	&	\vert\mathcal E_2\vert\leq e^{-\eta_{cs}t}\int_0^{N_1}Ke^{\gamma\left(t-s\right)-\int_0^sz\left(\theta_r\omega\right)dr}\Big[\vert G\left(\theta_s\omega,v\left(s,\omega,x_1\right)+\overline\psi\left(s,\omega,\iota_1,x_1\right)\right)\\&-G\left(\theta_s\omega,v\left(s,\omega,x_2\right)+\overline\psi\left(s,\omega,\iota_2,x_2\right)\right)\vert+\vert G\left(\theta_s\omega,v\left(s,\omega,x_1\right)\right)-G\left(\theta_s\omega,v\left(s,\omega,x_2\right)\right)\vert\Big]ds\\&\leq KL\int_0^{N_1}e^{\left(\eta_{cs}-\gamma\right)\left(s-t\right)}\Big[\Vert\overline\psi\left(\cdot,\omega,\iota_1,x_1\right)-\overline\psi\left(\cdot,\omega,\iota_2,x_2\right)\Vert_{\mathcal C_{\eta_{cs}}}+2\Vert v\left(\cdot,\omega,x_1\right)-v\left(\cdot,\omega,x_2\right)\Vert_{\mathcal C_{\eta_{cs}}}\Big]ds.
	\end{split}
\end{equation*}
Since $v\left(s,\omega,x\right)$ is a solution of \eqref{eq4} which is continuous in $x$, there exists $\delta_1>0$ such that if $\left\Vert x_1-x_2\right\Vert<\delta_1$, we have
\begin{equation*}
	\sup_{0\leq t\leq N_1}\Vert v\left(t,\omega,x_1\right)-v\left(t,\omega,x_2\right)\Vert\leq\frac{\varepsilon}{14}\cdot\frac{\eta_{cs}-\gamma}{2KLe^{\left(\eta_{cs}-\gamma\right)N_1}}.
\end{equation*}
Then we get that if $t>N_1$ and $\left\Vert x_1-x_2\right\Vert<\delta_1$, we have
\begin{equation}\label{eq37}
	\vert\mathcal E_2\vert\leq\frac16\Vert\overline\psi\left(\cdot,\omega,\iota_1,x_1\right)-\overline\psi\left(\cdot,\omega,\iota_2,x_2\right)\Vert_{\mathcal C_{\eta_{cs}}}+\frac{\varepsilon}{14}.
\end{equation}
If $t\leq N_1$, choose the same $\delta_1$ and if $\left\Vert x_1-x_2\right\Vert<\delta_1$, we have
\begin{equation*}
\begin{split}
	\sup_{0\leq s\leq t}\Vert v\left(s,\omega,x_1\right)-v\left(s,\omega,x_2\right)\Vert&\leq\sup_{0\leq t\leq N_1}\Vert v\left(t,\omega,x_1\right)-v\left(t,\omega,x_2\right)\Vert\\&\leq\frac{\varepsilon}{14}\cdot\frac{\eta_{cs}-\gamma}{2KLe^{\left(\eta_{cs}-\gamma\right)N_1}}\leq\frac{\varepsilon}{14}\cdot\frac{\eta_{cs}-\gamma}{2KL\left(1-e^{-\left(\eta_{cs}-\gamma\right)t}\right)},
	\end{split}
\end{equation*} 
still yielding \eqref{eq37}. Similarly, we can choose $N_2$ so positive that
\begin{equation}\label{eq38}
	\sup_{t\geq0}\vert\mathcal E_3\vert\leq\frac{\varepsilon}{14}.
\end{equation}
For such $N_2$, we can choose a $\delta_2>0$ so that if $\Vert x_1-x_2\Vert\leq\delta_2$,
\begin{equation}\label{eq39}
	\vert\mathcal E_4\vert\leq\frac16\Vert\overline\psi\left(\cdot,\omega,\iota_1,x_1\right)-\overline\psi\left(\cdot,\omega,\iota_2,x_2\right)\Vert_{\mathcal C_{\eta_{cs}}}+\frac{\varepsilon}{14}.
\end{equation}
For $t>N_3$,and choose $\Vert\iota_1-\iota_2\Vert<1$, \begin{equation*}
	\begin{split}
		&\left\vert{\mathcal E_5}\right\vert=e^{-\eta_{cs}t-\int_{0}^tz\left(\theta_r\omega\right)dr}\bigg\vert\lim_{\lambda\to+\infty}\int_{0}^{t-N_3}\phi_{A_{0s}}\left(t,l+N_3\right)\lambda{\left(\lambda I-A_s\right)^{-1}}\Pi_s\Delta\mathcal G\left(l+N_3\right)dl\bigg\vert\\&\leq LC_{\chi+\sigma}\underset{h\in[N_3,t]}{\sup}e^{\left(\chi+\sigma\right)\left(t-h\right)}e^{-\eta_{cs}t}e^{\left(\eta_{cs}-2\sigma\right)h}[\Vert \overline\psi\left(h,\omega,\iota_1,x_1\right)\Vert_{\mathcal C_{\eta_{cs}-2\sigma}}+\Vert\overline\psi\left(h,\omega,\iota_2,x_2\right)\Vert_{\mathcal C_{\eta_{cs}-2\sigma}}]\\&\leq LC_{\chi+\sigma}\underset{h\in[N_3,t]}{\sup}e^{\left(\chi+2\sigma-\eta_{cs}\right)\left(t-h\right)}e^{-\sigma t}[\Vert \overline\psi\left(\cdot,\omega,\iota_1,x_1\right)\Vert_{\mathcal C_{\eta_{cs}-2\sigma}}+\Vert\overline\psi\left(\cdot,\omega,\iota_2,x_2\right)\Vert_{\mathcal C_{\eta_{cs}-2\sigma}}]\\&\leq \frac{2KLC_{\chi+\sigma}e^{-\sigma N_3}}{1-KL\left(C_{\chi+2\sigma}+\frac1{\eta_{cs}-2\sigma-\gamma}+\frac1{\alpha-\eta_{cs}+2\sigma}\right)}\left(2\left\Vert{\iota_1}\right\Vert+1\right)\\&\leq\frac{12}{5}LC_{\chi+\sigma}e^{-\sigma N_3}\left(2\left\Vert{\iota_1}\right\Vert+1\right).	\end{split}
\end{equation*}
For given $\varepsilon>0$, choose $N_3$ so positive that
\begin{equation}\label{eq40}
	\sup_{t\geq0}\vert\mathcal E_5\vert\leq\frac{\varepsilon}{14}.
\end{equation}
Fix such $N_3$, for $t>N_3$,
\begin{equation*}
	\begin{split}
	&	\left\vert{\mathcal E_6}\right\vert\leq LC_{\chi+\sigma}\underset{s\in[0,N_3]}{\sup}e^{\left(\chi+\sigma-\eta_{cs}\right)\left(t-s\right)}\Big[\vert G\left(\theta_s\omega,v\left(s,\omega,x_1\right)+\overline\psi\left(s,\omega,\iota_1,x_1\right)\right)\\&-G\left(\theta_s\omega,v\left(s,\omega,x_2\right)+\overline\psi\left(s,\omega,\iota_2,x_2\right)\right)\vert+\vert G\left(\theta_s\omega,v\left(s,\omega,x_1\right)\right)-G\left(\theta_s\omega,v\left(s,\omega,x_2\right)\right)\vert\Big]\\&\leq KLC_{\chi+\sigma}\Big[\Vert\overline\psi\left(\cdot,\omega,\iota_1,x_1\right)-\overline\psi\left(\cdot,\omega,\iota_2,x_2\right)\Vert_{\mathcal C_{\eta_{cs}}}+2\Vert v\left(\cdot,\omega,x_1\right)-v\left(\cdot,\omega,x_2\right)\Vert_{\mathcal C_{\eta_{cs}}}\Big].	\end{split}
\end{equation*}
There exists $\delta_3>0$ such that if $\Vert x_1-x_2\Vert<\delta_3$, we have
\begin{equation*}
	\sup_{0\leq t\leq N_3}\Vert v\left(t,\omega,x_1\right)-v\left(t,\omega,x_2\right)\Vert\leq\frac{\varepsilon}{14}\cdot\frac{1}{2KLC_{\chi+\sigma}}.
\end{equation*}
Then we get that if $t>N_3$ and $\left\Vert x_1-x_2\right\Vert<\delta_3$, we have
\begin{equation}\label{eq41}
	\vert\mathcal E_6\vert\leq\frac16\Vert\overline\psi\left(\cdot,\omega,\iota_1,x_1\right)-\overline\psi\left(\cdot,\omega,\iota_2,x_2\right)\Vert_{\mathcal C_{\eta_{cs}}}+\frac{\varepsilon}{14}.
\end{equation}
If $t\leq N_3$, choose the same $\delta_3$ and if $\left\Vert x_1-x_2\right\Vert<\delta_3$, we have
\begin{equation*}
\begin{split}
	\sup_{0\leq s\leq t}\Vert v\left(s,\omega,x_1\right)-v\left(s,\omega,x_2\right)\Vert&\leq\sup_{0\leq t\leq N_3}\Vert v\left(t,\omega,x_1\right)-v\left(t,\omega,x_2\right)\Vert\leq\frac{\varepsilon}{14}\cdot\frac{1}{2KLC_{\chi+\sigma}},
	\end{split}
\end{equation*} 
still yielding \eqref{eq41}. And there exists $\delta_0>0$, such that if $\Vert\iota_1-\iota_2\Vert<\delta_0$,
\begin{equation}\label{eq42}
	\sup_{t\geq0}e^{-\eta_{cs}t-\int_0^tz\left(\theta_r\omega\right)dr}\Vert\phi_{A_{0c}}\left(t,0\right)\left(\iota_1-\iota_2\right)\Vert\leq\sup_{t\geq0}e^{\left(\gamma-\eta_{cs}\right)t}\Vert\iota_1-\iota_2\Vert\leq\frac{\varepsilon}{14}.
\end{equation}
Take $\delta=\min\{\delta_0,\delta_1,\delta_2,\delta_3,1\}$, from \eqref{eq36}-\eqref{eq42}, we have if $\Vert x_1-x_2\Vert<\delta$, $\Vert\iota_1-\iota_2\Vert<\delta$, then 
\begin{equation*}
	\Vert\overline\psi\left(\cdot,\omega,\iota_1,x_1\right)-\overline\psi\left(\cdot,\omega,\iota_2,x_2\right)\Vert_{\mathcal C_{\eta_{cs}}}\leq\frac12\Vert\overline\psi\left(\cdot,\omega,\iota_1,x_1\right)-\overline\psi\left(\cdot,\omega,\iota_2,x_2\right)\Vert_{\mathcal C_{\eta_{cs}}}+\frac{\varepsilon}{2},
\end{equation*}
which means
\begin{equation*}
	\Vert\overline\psi\left(\cdot,\omega,\iota_1,x_1\right)-\overline\psi\left(\cdot,\omega,\iota_2,x_2\right)\Vert_{\mathcal C_{\eta_{cs}}}\leq\varepsilon.
\end{equation*}
Therefore $\overline\psi\left(\cdot,\omega,\iota,x\right)$ is continuous with respect to $\left(\iota,x\right)$. Since $\overline{\psi}\left(\cdot,\omega,\iota,x\right)$ can be an $\omega$-wise limit of the iteration of a contraction mapping $\mathcal{Z}$ starting at 0 and $\mathcal{Z}$  maps a measurable function to a measurable function, $\overline{\psi}\left(\cdot,\omega,\iota,x\right)$ is measurable. By \cite[Lemma III.14]{Cas1977}, $\overline{\psi}$ is measurable with respect to $\left(\omega, \iota,x\right)$.

\noindent {\bf Step 3.} We will prove that the center-stable foliation is given by the graph of a Lipschitz continuous map. 

Let $l^{cs}\left(\iota,\omega,x\right)=\Pi_{0u}x+\Pi_u{\overline{\psi}\left(0,\omega,\iota-\Pi_{0cs}x,x\right)}$, then by \eqref{eq33}, 
\begin{equation*}
\begin{split}
	l^{cs}\left(\iota,\omega,x\right)=&\Pi_{0u}x-\int_{0}^{+\infty}\phi_{A_{0u}}\left(0,s\right)\Pi_u\Delta G\left(\theta_s\omega,v\left(s\right),\overline\psi\left(s\right)\right)ds.\end{split}	
\end{equation*}
Thus $l^{cs}\left(\iota,\omega,x\right)$ is measurable in $\left(\iota,\omega,x\right)$. Indeed, by \eqref{eq35}, it follows that
\begin{equation*}
	\left\Vert{l^{cu}\left(\iota_1,\omega,x\right)-l^{cu}\left(\iota_2,\omega,x\right)}\right\Vert\leq\frac{K^2L}{\left(\alpha-\eta_{cs}\right)\left(1-KL\left(C_\chi+\frac1{\eta_{cs}-\gamma}+\frac1{\alpha-\eta_{cs}}\right)\right)}\left\Vert{\iota_1-\iota_2}\right\Vert.
\end{equation*}
From the definition of $l^{cu}$ and the claim in {\bf Step 1}, it follows that $\tilde{x}\in{\mathcal{W}^{cs}\left(x,\omega\right)}$ if and only if there exists a $\iota\in{X_{0cs}}$ such that
\begin{equation*}
	\tilde{x}=x+\iota+\Pi_u\overline\psi\left(0,\omega,\iota,x\right).
\end{equation*}
Let $y=\Pi_{0cs}x+\iota$, then 
\begin{equation}\label{eq43}
	\tilde{x}=\Pi_{0u}x+y+\Pi_u\overline\psi\left(0,\omega,y-\Pi_{0cs}x,x\right).
\end{equation}
 replacing $y$ by $\iota$ in \eqref{eq43}, we get
 \begin{equation*}
 	\tilde{x}=\iota+l^{cs}\left(\iota,\omega,x\right).
 \end{equation*}
 Therefore, we have
\begin{equation*}
	\mathcal{W}^{cs}\left(x,\omega\right)=\{\iota+l^{cs}\left(\iota,\omega,x\right):\iota\in X_{0cs}\}.
\end{equation*}
\noindent {\bf Step 4.}We prove that $\mathcal W^{cs}\left(x,\omega\right)\cap\mathcal M^{cu}\left(\omega\right)$ contains a unique point.

If $\tilde x\in\mathcal W^{cs}\left(x,\omega\right)\cap\mathcal M^{cu}\left(\omega\right) $, there exists $\xi\in X_{0cu}$ and $\iota\in X_{0cs}$ such that 
\begin{equation*}
	\tilde x=\xi+h^{cu}\left(\xi,\omega\right)=\iota+l^{cs}\left(\iota,\omega,x\right),
\end{equation*}
which implies 
\begin{equation}\label{eq44}
	\xi=l^{cs}\left(\iota,\omega,x\right)~\text{and}~\iota=h^{cu}\left(\xi,\omega\right).
\end{equation}
By the Lipschitz continuity of $h^{cu}$ and $l^{cs}$ and $K_uK_s<1$, the equation
\begin{equation*}
	\iota=h^{cu}\left(l^{cs}\left(\iota,\omega,x\right),\omega\right)
\end{equation*}
has a unique solution. Thus $\mathcal W^{cs}\left(x,\omega\right)\cap\mathcal M^{cu}\left(\omega\right)$ contains a unique point. By \cite[Theorem 4.2]{KB2008} and \eqref{eq44}, we have 
\begin{equation*}
	\text{graph}~l^{cs}\left(\cdot,\omega,x\right)=\text{graph}~l^{cs}\left(\cdot,\omega,\xi\left(x\right)+h^{cu}\left(\xi\left(x\right),\omega\right)\right),
\end{equation*}
and $\text{graph}~l^{cs}\left(\cdot,\omega,x_1\right)=\text{graph}~l^{cs}\left(\cdot,\omega,x_2\right)$ if $\text{graph}~l^{cs}\left(\cdot,\omega,x_1\right)\cap\text{graph}~l^{cs}\left(\cdot,\omega,x_2\right)\neq\emptyset$. Hence $\cup_{\xi\in X_{0cu}}\text{graph}~l^{cs}\left(\iota,\omega,\xi+h^{cu}\left(\omega,\xi\right)\right)$ is a Lipschitz foliation.

\noindent {\bf Step 5.}
Finally, we prove $\mathcal{W}^{cs}\left(x,\omega\right)$ is invariant. 

Following Definition \ref{def2.2}, we need to show that for each $\tilde{x}\in{\mathcal{W}^{cs}\left(x,\omega\right)}$, $v\left(r,\omega,\tilde{x}\right)\in{\mathcal{W}^{cs}\left(v\left(r,\omega,x\right),\theta_r\omega\right)}$ for all $r\geq0$. By the cocycle property, for any $\nu\in X_0$, $v\left(t+r,\omega,\nu\right)=v\left(t,\theta_r\omega,v\left(r,\omega,\nu\right)\right)$. $v\left(t+r,\omega,\tilde{x}\right)-v\left(t+r,\omega,x\right)\in{\mathcal{C}_{\eta_{cs}}\left([0,+\infty);X_0\right)}$, we have $v\left(t,\theta_r\omega,v\left(r,\omega,\tilde{x}\right)\right)-v\left(t,\theta_r\omega,v\left(r,\omega,x\right)\right)\in{\mathcal{C}_{\eta_{cs}}\left([0,+\infty);X_0\right)}$. By the definition of $\mathcal{W}^{cs}\left(x,\omega\right)$, we have $v\left(r,\omega,\tilde{x}\right)\in{\mathcal{W}^{cs}\left(v\left(r,\omega,x\right),\theta_r\omega\right)}$.

\end{proof}
\begin{theorem}\label{th3.4}
Let $v\left(t,\omega,x\right)$ be a solution of \eqref{eq4} and $u\left(t,\omega,x\right)$ be the solution of \eqref{eq1}, then $\mathcal{W}^{cs*}\left(x,\omega\right)=\Xi^{-1}\left(\mathcal{W}^{cs}\left(x,\omega\right),\omega\right)$ is a leaf of invariant center-stable foliations for the random dynamical system generated by \eqref{eq1}.\end{theorem}
\begin{proof}
In fact, it follows from Lemma \ref{lem3.3} that
\begin{equation*}
\begin{split}
	u\left(t,\omega,\mathcal{W}^{cs*}\left(x,\omega\right)\right)&=\Xi^{-1}\left(\theta_t\omega,v\left(t,\omega,\Xi\left(\omega,\mathcal{W}^{cs*}\left(x,\omega\right)\right)\right)\right)\\&=\Xi^{-1}\left(\theta_t\omega,v\left(t,\omega,\mathcal{W}^{cs}\left(x,\omega\right)\right)\right)\subset{\Xi}^{-1}\left(\theta_t\omega,\mathcal{W}^{cs}\left(x,\theta_t\omega\right)\right)\\&=\mathcal{W}^{cs*}\left(x,\theta_t\omega\right).
\end{split}
\end{equation*}	
So $\mathcal{W}^{cs*}\left({x,\omega}\right)$ is invariant. Note that for any $x_i\in \mathcal{W}^{cs*}\left({x,\omega}\right)$, there exists $y_i\in \mathcal{W}^{cs}\left({x,\omega}\right)$ such that $x_i=e^{z\left(\omega\right)}y_i$ for $i=1,2$. Then $u\left(t,\omega,x_1\right)-u\left(t,\omega,x_2\right)=e^{z\left(\omega\right)}\left(v\left(t,\omega,y_1\right)-v\left(t,\omega,y_2\right)\right)$. From Lemma \ref{lemma2.1}, we have that $t\mapsto z\left(\omega\right)$ has a sublinear growth rate, thus the transform $\Xi^{-1}$ does not change the exponential convergence of solutions starting at the same leaf.  Moreover, for $t\geq0$,
\begin{equation*}
	\begin{split}
		\mathcal{W}^{cs*}\left(x,\omega\right)&=\Xi^{-1}\left(\omega,\mathcal{W}^{cs}\right)\\&=\left\{u_0=\Xi^{-1}\left(\omega,\iota+l^{cs}\left(\iota,\omega,x\right)\right):\iota\in{X_{0cs}}\right\}\\&=\left\{u_0=e^{z\left(\theta_t\omega\right)}\left(\iota+l^{cs}\left(\iota,\omega,x\right)\right):\iota\in{X_{0cs}}\right\}\\&=\left\{u_0=\iota+e^{z\left(\theta_t\omega\right)}l^{cs}\left(e^{-z\left(\theta_t\omega\right)}\iota,\omega,x\right):\iota\in{X_{0cs}}\right\}.
	\end{split}
\end{equation*}
 Therefore, $\mathcal{W}^{cs*}\left(x,\omega\right)$ is a Lipschitz center invariant foliation given by the graph of a Lipschitz continuous function $l^{cs*}\left(\iota,\omega,x\right)=e^{z\left(\theta_t\omega\right)}l^{cs}\left(e^{-z\left(\theta_t\omega\right)}\iota,\omega,x\right)$ over $X_{0cs}$.
\end{proof}

By similar arguments, we could get the existence of a center-unstable foliation $\mathcal W^{cu*}\left(x,\omega\right)$ of \eqref{eq1}. Then there is a center foliation $\mathcal W^c\left(x,\omega\right)$, which is obtained by an intersection of the other two: $\mathcal W^c=\mathcal W^{cs*}\cap\mathcal W^{cu*}$. Next we discuss the smoothness of $\mathcal{W}^{cs*}\left(x,\omega\right)$. We have the following result.
 \begin{theorem}\label{thm4.4}
	Assume $F$ is $C^k$ in $u$. For $-\beta<\chi<k\eta_{cs}$ and $\gamma<k\eta_{cs}<\min\{\alpha,\beta\}$ with 
\begin{equation}\label{eq45}
	{KL\left(C_{\chi+\left(i-1\right)\eta_{cs}}+\frac1{i\eta_{cs}-\gamma}+\frac1{\alpha-i\eta_{cs}}\right)}<1,~\text{for all} ~1\leq{i}\leq{k}
\end{equation}
then $l^{cs}\left(\iota,\omega,x\right)$ is $C^k$ in $\iota$ for $x\in X_0$.
\end{theorem}
\begin{proof}
	\noindent {\bf Step 1.} We prove that the fixed point $\overline{\psi}\left(t,\omega,\iota,x\right)$ of the Lyapunov-Perron operator $\mathcal{Z}\left(\cdot,\iota\right)$ defined in Theorem \ref{thm4.2} is $C^1$. For notational simplify, we denote $\mathcal{C}_{\eta_{cs}}\left([0,+\infty);X_0\right)$ by $\mathcal{C}_{\eta_{cs}}$ from now on. Notice that for $\gamma<\eta_{cs}^{'}<\eta_{cs}^{''}<\min\left\{\alpha,\beta\right\}$, $\mathcal{C}_{\eta_{cs}^{'}}\subset{\mathcal{C}_{\eta_{cs}^{''}}}$, thus a fixed point in $\mathcal{C}_{\eta_{cs}^{'}}$ must be in $\mathcal{C}_{\eta_{cs}^{''}}$. So $\overline{\psi}\left(\cdot,\omega,\iota,x\right)$ is independent of $\eta$ satisfying \eqref{eq38}. 

Clearly, there exits $\nu^*>0$ such that for $0<\nu\leq\nu^*$, we have $-\beta<\chi<i\eta_{cs}-2\nu<i\eta_{cs}-\nu	
$, $\gamma<i\eta_{cs}-2\nu<i\eta_{cs}-\nu<\min\left\{\alpha,\beta\right\}
$ and
\begin{equation}\label{eq46}
	KL\left(C_{\chi+j\nu}+\frac1{i\eta_{cs}-j\nu-\gamma}+\frac1{\alpha-i\eta_{cs}+j\nu}\right)<1,~\text{for}~\text{all}~1\leq{i}\leq{k},~j=1,2.
\end{equation}
Then $\mathcal{Z}\left(\cdot,\iota\right)$ is a uniform contraction in $\mathcal{C}_{\eta_{cs}-\nu}\subset{\mathcal{C}_{\eta_{cs}}}$. By \eqref{eq34}, for $\iota\in X_{0cs}$,
\begin{equation*}
\begin{split}
	\overline\psi\left(t,\omega,\iota,x\right)=&\phi_{A_{0c}}\left(t,0\right)\iota+\Pi_{0s}\left(S\diamond \Delta G_x\left(v,\overline\psi\right)\right)\left(t\right)\\&+\int_0^t\phi_{A_{0c}}\left(t,s\right)\Pi_c \Delta G\left(\theta_s\omega,v\left(s\right),\overline\psi\left(s,\omega,\iota,x\right)\right)ds\\&-\int_t^{+\infty}\phi_{A_{0u}}\left(t,s\right)\Pi_u \Delta G\left(\theta_s\omega,v\left(s\right),\overline\psi\left(s,\omega,\iota,x\right)\right)ds.
	\end{split}
\end{equation*}
By \eqref{eq35}, for $\iota,\iota_0\in X_{0cs}$ and $i=1,2$, $\overline\psi$ satisfies 
\begin{equation*}
	{\left\Vert{\overline{\psi}\left(t,\omega,\iota,x\right)-\overline{\psi}\left(t,\omega,\iota_0,x\right)}\right\Vert}_{\mathcal{C}_{{\eta_{cs}}-i \nu}}\leq\frac{K\left\Vert{\iota_1-\iota_2}\right\Vert}{1-KL\left(C_{\chi+i\nu}+\frac1{\eta_{cs}-i\nu-\gamma}+\frac1{\alpha-\eta_{cs}+i\nu}\right)}.
\end{equation*}
Fix $\nu\in(0,\nu^*]$, we first prove that $\overline{\psi}$ is differentialble from $X_{0cs}$ to $\mathcal{C}_{\eta_{cs}-\nu}$. For the solution of \eqref{eq4} $v\left(t\right)$, $\iota_0\in{X_{0cs}}$, $\psi\in{\mathcal{C}_{\eta-\nu}}$, define $\mathcal{T}:\mathcal{C}_{\eta_{cs}-\nu}\to{\mathcal{C}_{\eta_{cs}-\nu}}$ as follows
\begin{equation*}
\begin{split}
	\mathcal{T}\psi=&\int_0^t\phi_{A_{0c}}\left(t,s\right)\Pi_c D_\psi\Delta G\left(\theta_s\omega,v\left(s\right),\overline\psi\left(s,\omega,\iota_0,x\right)\right)\psi\left(s\right)ds\\&-\int_t^{+\infty}\phi_{A_{0u}}\left(t,s\right)\Pi_u D_\psi\Delta G\left(\theta_s\omega,v\left(s\right),\overline\psi\left(s,\omega,\iota_0,x\right)\right)\psi\left(s\right)ds\\&+\lim_{\lambda\to+\infty}\int_{0}^t\phi_{A_{0s}}\left(t,s\right)\lambda{\left(\lambda{I}-A_s\right)}^{-1}\\&\times\Pi_s D_\psi\Delta G\left(\theta_s\omega,v\left(s\right),\overline\psi\left(s,\omega,\iota_0,x\right)\right)\psi\left(s\right)ds.
	\end{split}
\end{equation*}
Using the same arguments as we proved that $\mathcal{Z}$ is a contraction and $\left\Vert{D_\psi \Delta G\left(\theta_s\omega,v,\psi\right)}\right\Vert\leq{Lip~G=L}$, we have $\mathcal{T}$ is a bounded linear operator from $\mathcal{C}_{\eta_{cs}-\nu}$ to itself with the norm
\begin{equation*}
	\left\Vert{\mathcal{T}}\right\Vert\leq{K}L\left(C_{\chi+\nu}+\frac1{\eta_{cs}-\nu-\gamma}+\frac1{\alpha-\eta_{cs}+\nu}\right)<1.
\end{equation*}
Thus $id-\mathcal{T}$ is invertible in $\mathcal{C}_{\eta_{cs}-\nu}$. For $\iota,\iota_0\in{X_{0cs}}$, set
\begin{equation*}
	\begin{split}
		\mathcal S\left(t\right)&=\int_0^t\phi_{A_{0c}}\left(t,s\right)\Pi_c\overline\Delta_{G\left(\psi\right)}\left(s\right)ds-\int_t^{+\infty}\phi_{A_{0u}}\left(t,s\right)\Pi_u\overline\Delta_{G\left(\psi\right)}\left(s\right)ds\\&+\lim_{\lambda\to+\infty}\int_{0}^t\phi_{A_{0s}}\left(t,s\right)\lambda{\left(\lambda{I}-A_s\right)}^{-1}\Pi_s \overline\Delta_{G\left(\psi\right)}\left(s\right)ds
   \end{split}
\end{equation*}
where 
\begin{equation*}
\begin{split}
	\overline\Delta_{G\left(\psi\right)}\left(s\right)&=\Delta G\left(\theta_s\omega,v\left(s\right),\overline\psi\left(s,\omega,\iota,x\right)\right)-\Delta G\left(\theta_s\omega,v\left(s\right),\overline\psi\left(s,\omega,\iota_0,x\right)\right)\\&-D_\psi\Delta G\left(\theta_s\omega,v\left(s\right),\overline\psi\left(s,\omega,\iota_0,x\right)\right)\left(\overline\psi\left(s,\omega,\iota,x\right)-\overline\psi\left(s,\omega,\iota_0,x\right)\right)\\&=G\left(\theta_s\omega,v\left(s\right)+\overline\psi\left(s,\omega,\iota,x\right)\right)-G\left(\theta_s\omega,v\left(s\right)+\overline\psi\left(s,\omega,\iota_0,x\right)\right)\\&-D_\psi G\left(\theta_s\omega,v\left(s\right)+\overline\psi\left(s,\omega,\iota_0,x\right)\right)\left(\overline\psi\left(s,\omega,\iota,x\right)-\overline\psi\left(s,\omega,\iota_0,x\right)\right).
	\end{split}
\end{equation*}
We prove $\left\Vert{\mathcal S}\right\Vert_{\mathcal{C}_{\eta_{cs}-\nu}}=o\left(\left\Vert\iota-\iota_0\right\Vert\right)$ as $\iota\to\iota_0$, then we have 
\begin{equation*}
\begin{split}
\overline{\psi}\left(\cdot,\omega,\iota,x\right)-\overline{\psi}\left(\cdot,\omega,\iota_0,x\right)&-\mathcal{T}\left(\overline{\psi}\left(\cdot,\omega,\iota,x\right)-\overline{\psi}\left(\cdot,\omega,\iota_0,x\right)\right)\\&=\phi_{A_0}\left(t,0\right)\left(\iota-\iota_0\right)+\mathcal S\\&=	\phi_{A_0}\left(t,0\right)\left(\iota-\iota_0\right)+o\left(\left\Vert\iota-\iota_0\right\Vert\right),
\end{split}
\end{equation*}
which implies
\begin{equation*}
	\overline{\psi}\left(\cdot,\omega,\iota,x\right)-\overline{\psi}\left(\cdot,\omega,\iota_0,x\right)=\left(id-\mathcal{T}\right)^{-1}\phi_{A_0}\left(t,0\right)\left(\iota-\iota_0\right)+o\left(\left\Vert\iota-\iota_0\right\Vert\right).
\end{equation*}
We know that $\phi_{A_0}\left(t,0\right)=T\left(t\right)e^{\int_{0}^{t}z\left(\theta_r\omega\right)dr}$ is a bounded operator, therefore $\overline{\psi}\left(\cdot,\omega,\iota,x\right)$ is differentiable in $\iota$ and $D_\iota\overline{\psi}\left(\cdot,\omega,\iota,x\right)\in{{\mathcal L}}\left(X_{0cs},\mathcal{C}_{{\eta_{cs}}-\nu}\right)$. Next we prove $\left\Vert{\mathcal S}\right\Vert_{\mathcal{C}_{\eta_{cs}-\nu}}=o\left(\left\Vert\iota-\iota_0\right\Vert\right)$ as $\iota\to\iota_0$. We write $e^{-\left(\eta_{cs}-\nu\right)t-\int_0^tz\left(\theta_r\omega\right)dr}\mathcal S=\sum_{i=1}^6\mathcal S_{i}$, where for $t>N_1$,
\begin{equation*}
	\mathcal S_1=e^{-\left(\eta_{cs}-\nu\right)t-\int_{0}^tz\left(\theta_r\omega\right)dr}\int_{N_1}^t\phi_{A_{0c}}\left(t,s\right)\Pi_c\overline\Delta_{G\left(\psi\right)}\left(s\right)ds,
\end{equation*}
and $\mathcal S_1=0$ for $t\leq{N_1}$; For $t>N_1$,
\begin{equation*}
	\mathcal S_2=e^{-\left(\eta_{cs}-\nu\right)t-\int_{0}^tz\left(\theta_r\omega\right)dr}\int_{0}^{N_1}\phi_{A_{0c}}\left(t,s\right)\Pi_c\overline\Delta_{G\left(\psi\right)}\left(s\right)ds,
\end{equation*}
and change $N_1$ to $t$ if $t\leq{N_1}$. For $t>N_2$,
\begin{equation*}
	\mathcal S_3=-e^{-\left(\eta_{cs}-\nu\right)t-\int_{0}^tz\left(\theta_r\omega\right)dr}\int_{t}^{N_2}\phi_{A_{0u}}\left(t,s\right)\Pi_u\overline\Delta_{G\left(\psi\right)}\left(s\right)ds,
\end{equation*}
and $\mathcal S_3=0$ for $t\leq{N_2}$; For $t<N_2$,
\begin{equation*}
	\mathcal S_4=-e^{-\left(\eta_{cs}-\nu\right)t-\int_{0}^tz\left(\theta_r\omega\right)dr}\int_{N_2}^{+\infty}\phi_{A_{0u}}\left(t,s\right)\Pi_u\overline\Delta_{G\left(\psi\right)}\left(s\right)ds,
\end{equation*}
and change $N_2$ to $t$ if $t\geq{N_2}$. For $t>N_3$,
\begin{equation*}
	\mathcal S_5=e^{-\left(\eta_{cs}-\nu\right)t-\int_{0}^tz\left(\theta_r\omega\right)dr}\lim_{\lambda\to+\infty}\int_{N_3}^t\phi_{A_{0s}}\left(t,s\right)\lambda{\left(\lambda{I}-A_s\right)}^{-1}\Pi_s\overline\Delta_{G\left(\psi\right)}\left(s\right)ds,
\end{equation*}
and $\mathcal S_5=0$ for $t\leq{N_3}$; For $t>N_3$,
\begin{equation*}
	\mathcal S_6=e^{-\left(\eta_{cs}-\nu\right)t-\int_{0}^tz\left(\theta_r\omega\right)dr}\lim_{\lambda\to+\infty}\int_{0}^{N_3}\phi_{A_{0s}}\left(t,s\right)\lambda{\left(\lambda{I}-A_s\right)}^{-1}\Pi_s\overline\Delta_{G\left(\psi\right)}\left(s\right)ds,
\end{equation*}
and change $N_3$ to $t$ if $t\leq{N_3}$. $N_1$, $N_2$, $N_3$ above are large positive numbers to be chosen later. By \eqref{eq5} and \eqref{eq35},
for $t>N_1$,
\begin{equation*}
	\begin{split}
	\left\vert{\mathcal S_1}\right\vert&\leq2KL\int_{N_1}^te^{\gamma\left(t-s\right)}e^{-\left(\eta_{cs}-\nu\right)t}e^{\left(\eta_{cs}-2\nu\right)s}\left\Vert\overline{\psi}\left(\cdot,\omega,\iota,x\right)-\overline{\psi}\left(\cdot,\omega,\iota_0,x\right)\right\Vert_{\mathcal{C}_{\eta_{cs}-2\nu}}ds\\&\leq2KL\left\Vert\overline{\psi}\left(\cdot,\omega,\iota,x\right)-\overline{\psi}\left(\cdot,\omega,\iota_0,x\right)\right\Vert_{\mathcal{C}_{\eta_{cs}-2\nu}}e^{\left(\gamma-\eta_{cs}+\nu\right)t}	\int_{N_1}^te^{\left(\eta_{cs}-2\nu-\gamma\right)s}ds\\&\leq	\frac{2K^2L}{\left(\eta_{cs}-2\nu-\gamma\right)\left[1-KL\left(C_{\zeta+2\nu}+\frac1{\eta_{cs}-2\nu-\gamma}+\frac1{\alpha-\eta_{cs}+2\nu}\right)\right]}e^{-\nu{N_1}}\left\Vert\iota-\iota_0\right\Vert.
	\end{split}
\end{equation*} 
Choose $N_1$ so large that 
\begin{equation}\label{eq47}
	\underset{t\geq0}{\sup}\left\vert{\mathcal S_1}\right\vert\leq\frac{\varepsilon}{6}\left\Vert\iota-\iota_0\right\Vert.
\end{equation}
Fix such $N_1$, 
\begin{equation*}
\begin{split}
	&\left\vert{\mathcal S_2}\right\vert\leq{K}\left\Vert\overline{\psi}\left(\cdot,\omega,\iota,x\right)-\overline{\psi}\left(\cdot,\omega,\iota_0,x\right)\right\Vert_{\mathcal{C}_{\eta_{cs}-\nu}}e^{\left(\gamma-\eta_{cs}+\nu\right)t}	\int_{0}^{N_1}e^{\left(\eta_{cs}-\nu-\gamma\right)s}\\&\times\int_0^1\Big\Vert{D}_\psi G\left(\theta_s\omega,v\left(s\right)+\tau\overline{\psi}\left(s,\omega,\iota,x\right)+\left(1-\tau\right)\overline{\psi}\left(s,\omega,\iota_0,x\right)\right)\\&-D_\psi G\left(\theta_s\omega,v\left(s\right)+\overline{\psi}\left(s,\omega,\iota_0,x\right)\right)\Big\Vert{d\tau}ds.
\end{split}
\end{equation*}
Since $G\left(\omega,v\right)$ is $C^k$, $D_\psi G\left(\omega,v,\psi\right)$ is continuous. For $\varepsilon>0$, there exists $\rho_1>0$ such that if $\left\Vert\iota-\iota_0\right\Vert<\rho_1$,
\begin{equation*}
\begin{split}
	&\Big\Vert{D}_\psi G\left(\theta_s\omega,v\left(s\right)+\tau\overline{\psi}\left(s,\omega,\iota,x\right)+\left(1-\tau\right)\overline{\psi}\left(s,\omega,\iota_0,x\right)\right)\\&-D_\psi G\left(\theta_s\omega,v\left(s\right)+\overline{\psi}\left(s,\omega,\iota_0,x\right)\right)\Big\Vert\\&\leq\frac{\varepsilon\left(\eta_{cs}-\nu-\gamma\right)\left[1-KL\left(C_{\zeta+\nu}+\frac1{\eta_{cs}-\nu-\gamma}+\frac1{\alpha-\eta_{cs}+\nu}\right)\right]}{6K^2e^{\left(\eta_{cs}-\nu-\gamma\right)N_1}}.
\end{split}
\end{equation*}
Then by \eqref{eq35},
\begin{equation}\label{eq48}
	\underset{t\geq0}{\sup}\left\vert{\mathcal S_2}\right\vert\leq\frac{\varepsilon}{6}\left\Vert\iota-\iota_0\right\Vert ~\text{for} ~\left\Vert\iota-\iota_0\right\Vert<\rho_1.
\end{equation}
Similarly, by choosing $N_2$ to be sufficiently large, we have
\begin{equation}\label{eq49}
	\underset{t\geq0}{\sup}\left\vert{\mathcal S_3}\right\vert\leq\frac{\varepsilon}{6}\left\Vert\iota-\iota_0\right\Vert.
\end{equation}
Fix such $N_2$, there exists $\rho_2>0$ such that if $\left\Vert\iota-\iota_0\right\Vert<\rho_2$, 
\begin{equation}\label{eq50}
	\underset{t\geq0}{\sup}\left\vert{\mathcal S_4}\right\vert\leq\frac{\varepsilon}{6}\left\Vert\iota-\iota_0\right\Vert ~\text{for} ~\left\Vert\iota-\iota_0\right\Vert<\rho_2.
\end{equation} 
For $\mathcal S_5$, if $t>N_3$,
\begin{equation*}
	\begin{split}
		&\left\vert{\mathcal S_5}\right\vert=e^{-\left(\eta_{cs}-\nu\right)t-\int_{0}^tz\left(\theta_r\omega\right)dr}\lim_{\lambda\to+\infty}\int_{0}^{t-N_3}\phi_{A_{0s}}\left(t,l+N_3\right)\lambda{\left(\lambda{I}-A_s\right)}^{-1}\Pi_s\overline\Delta_{G\left(\psi\right)}\left(l+N_3\right)dl\\&\leq2LC_{\chi+2\nu}\underset{h\in[N_3,t]}{\sup}e^{\left(\chi+4\nu-\eta_{cs}\right)\left(t-h\right)}e^{-\nu{t}}\left\Vert\overline{\psi}\left(\cdot,\omega,\iota,x\right)-\overline{\psi}\left(\cdot,\omega,\iota_0,x\right)\right\Vert_{\mathcal{C}_{\eta_{cs}-2\nu}}\\&\leq\frac{2KLC_{\chi+2\nu}e^{-\nu{N_3}}}{1-KL\left(C_{\chi+2\nu}+\frac1{\eta_{cs}-2\nu-\gamma}+\frac1{\alpha-\eta_{cs}+2\nu}\right)}\left\Vert\iota-\iota_0\right\Vert.
	\end{split}
\end{equation*}
Choose $N_3$ so large that 
\begin{equation}\label{eq51}
	\underset{t\geq0}{\sup}\left\vert{\mathcal S_5}\right\vert\leq\frac{\varepsilon}{6}\left\Vert\iota-\iota_0\right\Vert.
\end{equation}
And for the last one $\mathcal S_{6}$,
\begin{equation*}
	\begin{split}
		&\left\vert{\mathcal S_{6}}\right\vert\leq LC_{\chi+\nu}\sup_{s\in[0,N_3]}e^{\left(\chi+2\nu-\eta_{cs}\right)\left(t-s\right)}\left\Vert\overline{\psi}\left(\cdot,\omega,\iota,x\right)-\overline{\psi}\left(\cdot,\omega,\iota_0,x\right)\right\Vert_{\mathcal{C}_{\eta_{cs}-\nu}}\\&\times\int_0^1\Big\Vert{D}_\psi G\left(\theta_s\omega,v\left(s\right)+\tau\overline{\psi}\left(s,\omega,\iota,x\right)+\left(1-\tau\right)\overline{\psi}\left(s,\omega,\iota_0,x\right)\right)\\&-D_\psi G\left(\theta_s\omega,v\left(s\right)+\overline{\psi}\left(s,\omega,\iota_0,x\right)\right)\Big\Vert{d\tau}.
   \end{split}
\end{equation*}
For $\varepsilon>0$, there exists $\rho_3>0$ such that if $\left\Vert\iota-\iota_0\right\Vert<\rho_3$,
\begin{equation*}
\begin{split}
	&\Big\Vert{D}_\psi G\left(\theta_s\omega,v\left(s\right)+\tau\overline{\psi}\left(s,\omega,\iota,x\right)+\left(1-\tau\right)\overline{\psi}\left(s,\omega,\iota_0,x\right)\right)\\&-D_\psi G\left(\theta_s\omega,v\left(s\right)+\overline{\psi}\left(s,\omega,\iota_0,x\right)\right)\Big\Vert\\&\leq\frac{1-KL\left(C_{\chi+\nu}+\frac1{\eta_{cs}-\nu-\gamma}+\frac1{\alpha-\eta_{cs}+\nu}\right)}{KLC_{\chi+\nu}}.
\end{split}
\end{equation*}
Then by \eqref{eq35},
\begin{equation}\label{eq52}
	\underset{t\geq0}{\sup}\left\vert{\mathcal S_6}\right\vert\leq\frac{\varepsilon}{6}\left\Vert\iota-\iota_0\right\Vert ~\text{for} ~\left\Vert\iota-\iota_0\right\Vert<\rho_3.
\end{equation}
By taking $\rho_0=\min\left\{\rho_1,\rho_2,\rho_3\right\}$ and combining \eqref{eq47}-\eqref{eq52}, we have
\begin{equation*}
	\left\Vert{\mathcal S}\right\Vert_{\mathcal{C}_{\eta_{cs}-\nu}}\leq\varepsilon\left\Vert\iota-\iota_0\right\Vert~\text{for}~\left\Vert\iota-\iota_0\right\Vert<\rho_0. 
\end{equation*}
This implies $\left\Vert{\mathcal S}\right\Vert_{\mathcal{C}_{\eta_{cs}-\nu}}=o\left(\left\Vert\iota-\iota_0\right\Vert\right)$ as $\iota\to\iota_0$. Next we prove $D_\iota\overline{\psi}\left(\cdot,\omega,\iota\right)$ is continuous with respect to $\iota$ from $X_{0cs}$ to 	${ \mathcal L}\left(X_{0cs},\mathcal{C}_{\eta}\right)$. By \eqref{eq31}, we have 
\begin{equation*}
\begin{split}
&D_\iota\overline\psi\left(t,\omega,\iota,x\right)=\phi_{A_{0c}}\left(t,0\right)\Pi_{0cs}\\&+\int_0^t\phi_{A_{0c}}\left(t,s\right)\Pi_c D_\psi  G\left(\theta_s\omega,v\left(s\right)+\overline\psi\left(s,\omega,\iota,x\right)\right)D_\iota\overline\psi\left(s,\omega,\iota,x\right)ds\\&-\int_t^{+\infty}\phi_{A_{0u}}\left(t,s\right)\Pi_u D_\psi G\left(\theta_s\omega,v\left(s\right)+\overline\psi\left(s,\omega,\iota,x\right)\right)D_\iota\overline\psi\left(s,\omega,\iota,x\right)ds\\&+\lim_{\lambda\to+\infty}\int_{0}^t\phi_{A_{0s}}\left(t,s\right)\lambda{\left(\lambda{I}-A_s\right)}^{-1}\\&\times\Pi_s D_\psi G\left(\theta_s\omega,v\left(s\right)+\overline\psi\left(s,\omega,\iota,x\right)\right)D_\iota\psi\left(s,\omega,\iota,x\right)ds.
\end{split}	
\end{equation*}
For $\iota\in X_{0cs}$ and $\psi\in\mathcal{C}_{\eta_{cs}}$, define the operator $\mathcal{T}^{'}:\mathcal C_{\eta_{cs}}\to{\mathcal C_{\eta_{cs}}}$ as follows(replace the $\iota_0$ with $\iota$ in the definition of $\mathcal{T}$)
\begin{equation*}
\begin{split}
	\mathcal{T}^{'}\psi=&\int_0^t\phi_{A_{0c}}\left(t,s\right)\Pi_c D_\psi\Delta G\left(\theta_s\omega,v\left(s\right),\overline\psi\left(s,\omega,\iota,x\right)\right)\psi\left(s\right)ds\\&-\int_t^{+\infty}\phi_{A_{0u}}\left(t,s\right)\Pi_u D_\psi\Delta G\left(\theta_s\omega,v\left(s\right),\overline\psi\left(s,\omega,\iota,x\right)\right)\psi\left(s\right)ds\\&+\lim_{\lambda\to+\infty}\int_{0}^t\phi_{A_{0s}}\left(t,s\right)\lambda{\left(\lambda{I}-A_s\right)}^{-1}\\&\times\Pi_s D_\psi\Delta G\left(\theta_s\omega,v\left(s\right),\overline\psi\left(s,\omega,\iota,x\right)\right)\psi\left(s\right)ds.
	\end{split}
\end{equation*}
Similarly, we have
\begin{equation}\label{eq53}
	\Vert{\mathcal{T^{'}}}\Vert\leq{K}L\left(C_{\chi}+\frac1{\eta_{cs}-\gamma}+\frac1{\alpha-\eta_{cs}}\right)<1.
\end{equation}
For $\iota,\iota_0\in{X_{0cs}}$, 
\begin{equation}\label{eq54}
	\begin{split}
		&D_\iota\overline{\psi}\left(t,\omega,\iota,x\right)-D_\iota\overline{\psi}\left(t,\omega,\iota_0,x\right)\\&=\int_0^t\phi_{A_{0c}}\left(t,s\right)\Pi_c\Big[D_\psi G\left(\theta_s\omega,v\left(s\right)+\overline{\psi}\left(s,\omega,\iota,x\right)\right)D_\iota\overline{\psi}\left(s,\omega,\iota,x\right)\\&-D_\psi G\left(\theta_s\omega,v\left(s\right)+\overline{\psi}\left(s,\omega,\iota_0,x\right)\right)D_\iota\overline{\psi}\left(s,\omega,\iota_0,x\right)\Big]ds\\&+\int_{+\infty}^{t}\phi_{A_{0u}}\left(t,s\right)\Pi_u\Big[D_\psi G\left(\theta_s\omega,v\left(s\right)+\overline{\psi}\left(s,\omega,\iota,x\right)\right)D_\iota\overline{\psi}\left(s,\omega,\iota,x\right)\\&-D_\psi G\left(\theta_s\omega,v\left(s\right)+\overline{\psi}\left(s,\omega,\iota_0,x\right)\right)D_\iota\overline{\psi}\left(s,\omega,\iota_0,x\right)\Big]ds\\&+\lim_{\lambda\to+\infty}\int_{0}^t\phi_{A_{0s}}\left(t,s\right)\lambda{\left(\lambda{I}-A_s\right)}^{-1}\Pi_s\Big[D_\psi G\left(\theta_s\omega,v\left(s\right)+\overline{\psi}\left(s,\omega,\iota,x\right)\right)\\&D_\iota\overline{\psi}\left(s,\omega,\iota,x\right)-D_\psi G\left(\theta_s\omega,v\left(s\right)+\overline{\psi}\left(s,\omega,\iota_0,x\right)\right)D_\iota\overline{\psi}\left(s,\omega,\iota_0,x\right)\Big]ds\\&=\mathcal{T}^{'}\left(D_\iota\overline{\psi}\left(t,\omega,\iota,x\right)-D_\iota\overline{\psi}\left(t,\omega,\iota_0,x\right)\right)+{\mathcal S}^{'},
	\end{split}
\end{equation}
where
\begin{equation*}
\begin{split}
{\mathcal S}^{'}&=\int_0^t\phi_{A_{0c}}\left(t,s\right)\Pi_c\Big[D_\psi G\left(\theta_s\omega,v\left(s\right)+\overline\psi\left(s,\omega,\iota,x\right)\right)\\&-D_\psi G\left(\theta_s\omega,v\left(s\right)+\overline{\psi}\left(s,\omega,\iota_0,x\right)\right)\Big]D_\iota\overline{\psi}\left(s,\omega,\iota_0,x\right)ds\\&+\int_{+\infty}^{t}\phi_{A_{0u}}\left(t,s\right)\Pi_u\Big[D_\psi G\left(\theta_s\omega,v\left(s\right)+\overline\psi\left(s,\omega,\iota,x\right)\right)\\&-D_\psi G\left(\theta_s\omega,v\left(s\right)+\overline{\psi}\left(s,\omega,\iota_0,x\right)\right)\Big]D_\iota\overline{\psi}\left(s,\omega,\iota_0,x\right)ds\\&+\lim_{\lambda\to+\infty}\int_{0}^t\phi_{A_{0s}}\left(t,s\right)\lambda{\left(\lambda{I}-A_s\right)}^{-1}\Pi_s\Big[D_\psi G\left(\theta_s\omega,v\left(s\right)+\overline\psi\left(s,\omega,\iota,x\right)\right)\\&-D_\psi G\left(\theta_s\omega,v\left(s\right)+\overline{\psi}\left(s,\omega,\iota_0,x\right)\right)\Big]D_\iota\overline{\psi}\left(s,\omega,\iota_0,x\right)ds.
\end{split}
\end{equation*}
By \eqref{eq53}, $id-\mathcal{T}^{'}$ has a bounded inverse in ${\mathcal L}\left(X_{0cs},\mathcal{C}_{\eta_{cs}}\right)$. From \eqref{eq54}, we have
\begin{equation*}
	D_\iota\overline{\psi}\left(t,\omega,\iota,x\right)-D_\iota\overline{\psi}\left(t,\omega,\iota_0,x\right)=\left(id-\mathcal{T}^{'}\right)^{-1}{\mathcal S}^{'}.
\end{equation*}
Then by \eqref{eq53},
\begin{equation}\label{eq55}
	\Vert{D_\iota\overline{\psi}\left(t,\omega,\iota,x\right)-D_\iota\overline{\psi}\left(t,\omega,\iota_0,x\right)}\Vert_{{{L}}\left(X_{0cs},\mathcal{C}_{\eta_{cs}}\right)}\leq\frac{{\left\Vert{{\mathcal S}^{'}}\right\Vert}_{{\mathcal L}\left(X_{0cs},\mathcal{C}_\eta\right)}}{1-{K}L\left(C_{\chi}+\frac1{\eta_{cs}-\gamma}+\frac1{\alpha-\eta_{cs}}\right)}.
\end{equation}
Using the same procedure as for $\mathcal S$, we have ${\Vert{{\mathcal S}^{'}}\Vert}_{{ L}\left(X_{0cs},\mathcal{C}_{\eta_{cs}}\right)}=o\left(1\right)$ as $\iota\to\iota_0$. Then by \eqref{eq55}, $D_\iota\overline{\psi}\left(\cdot,\omega,\iota,x\right)$ is continuous with respect to $\iota$. So we proved that $\overline{\psi}\left(t,\omega,\iota,x\right)$ is $C^1$. 

\noindent {\bf Step 2.} We prove $\overline{\psi}\left(t,\omega,\iota,x\right)$ is $C^k$. 

Make the inductive assumption that $\overline{\psi}\left(\cdot,\omega,\iota,x\right)$ is $C^j$ from $X_{0cs}$ to $\mathcal{C}_{j\eta_{cs}}$ for all $1\leq{j}\leq{k-1}$ and we prove it for $j=k$. By calculation, the $\left(k-1\right)$th-derivative $D^{k-1}_\iota\overline{\psi}\left(\cdot,\omega,\iota,x\right)$ satiefies the following equation
\begin{equation*}
	\begin{split}
		&D^{k-1}_\iota\overline{\psi}\left(t,\omega,\iota,x\right)=\int_0^t\phi_{A_{0c}}\left(t,s\right)\Pi_c D_\psi G\left(\theta_s\omega,v\left(s\right)+\overline{\psi}\left(s,\omega,\iota,x\right)\right)D^{k-1}_\iota\overline{\psi}\left(s,\omega,\iota,x\right)ds\\&-\int_t^{+\infty}\phi_{A_{0u}}\left(t,s\right)\Pi_u D_\psi G\left(\theta_s\omega,v\left(s\right)+\overline{\psi}\left(s,\omega,\iota,x\right)\right)D^{k-1}_\iota\overline{\psi}\left(s,\omega,\iota,x\right)ds\\&+\lim_{\lambda\to+\infty}\int_0^t\phi_{A_{0s}}\left(t,s\right)\lambda{\left(\lambda{I}-A_s\right)}^{-1}\Pi_s D_\psi G\left(\theta_s\omega,v\left(s\right)+\overline{\psi}\left(s,\omega,\iota,x\right)\right)D^{k-1}_\iota\overline{\psi}\left(s,\omega,\iota,x\right)ds\\&+\int_0^t\phi_{A_{0c}}\left(t,s\right)\Pi_c R\left(s,\omega,\iota,x\right)ds-\int_t^{+\infty}\phi_{A_{0u}}\left(t,s\right)\Pi_u R\left(s,\omega,\iota,x\right)ds\\&+\lim_{\lambda\to+\infty}\int_0^t\phi_{A_{0s}}\left(t,s\right)\lambda{\left(\lambda{I}-A_s\right)}^{-1}\Pi_s R\left(s,\omega,\iota,x\right)ds,
	\end{split}
\end{equation*}
where
\begin{equation*}
	R\left(s,\omega,\iota,x\right)=\sum_{i=1}^{k-3}\binom{k-2}{i}D^{k-2-i}_\iota\left(D_\psi G\left(\theta_s\omega,v\left(s\right)+\overline{\psi}\left(s,\omega,\iota,x\right)\right)\right)D^{i+1}_\iota\overline{\psi}\left(s,\omega,\iota,x\right).
\end{equation*}
Applying the chain rule to $D^{m-2-i}_\iota\left(D_\psi G\left(\theta_s\omega,v\left(s\right)+\overline{\psi}\left(s,\omega,\iota,x\right)\right)\right)$, we observe that each term in $R_{m-1}\left(s,\omega,\iota,x\right)$ contains factors: $D^{i_1}_\psi{G\left(\theta_s\omega,v\left(s\right)+\overline{\psi}\left(s,\omega,\iota,x\right)\right)}$ for some $2\leq{i_1}\leq{m-1}$ and at least two derivatives $D^{i_2}_\iota\overline{\psi}\left(s,\omega,\iota,x\right)$ and $D^{i_3}_\iota\overline{\psi}\left(s,\omega,\iota,x\right)$ for some $i_2,i_3\in\{1,...,m-2\}$. Since $D^{i}_\iota\overline{\psi}\left(\cdot,\omega,\iota,x\right)\in{\mathcal{C}_{i\eta_{cs}}}$ for $i=1,...,m-1$ and $F$ is $C^k$, $R_{m-1}\left(\cdot,\omega,\iota,x\right):X_{0cs}\to{{L}^{m-1}\left(X_{0cs},X_0\right)}$ are $C^1$ in $\iota$. Using the same argument we used in the case $k=1$, we can show that $D^{m-1}_\iota\overline{\psi}\left(\cdot,\omega,\iota,x\right)$ is $C^1$ from $X_{0cs}$ to ${L}^m\left(X_{0cs},\mathcal{C}_{m\eta_{cs}}\right)$. So $\overline{\psi}\left(t,\omega,\iota,x\right)$ is $C^k$. Thus $l^{cs}\left(\iota,\omega,x\right)$ is $C^k$, which gurantees the $C^k$-smoothness of $\mathcal{W}^{cs*}\left(x,\omega\right)$.
\end{proof}

\section{Illustrative examples}\label{sec5}
In this section, we give two examples about age-structured problems and stochastic parabolic equations. Note that we only verify the main assumptions in this article and do not provide specific results of invariant manifolds and foliations.

\subsection{Age-structured problems}
Consider the following age-structured model which is a hyperbolic stochastic partial differential equation
 \begin{equation}\label{eq5.2}
	\left\lbrace
\begin{split}
&\frac{\partial u}{\partial t}+\frac{\partial u}{\partial a}=-\mu u\left(t,a\right)+u\left(t,a\right)\circ W\left(t\right), t\geq t_0,~a\geq0,\\
&u\left(t,0\right)=\int_0^{+\infty}h\left(t,a\right)u\left(t,a\right)da,
\\
&u\left(t_0,\cdot\right)=\varphi\in L_+^p\left((0,+\infty);\mathbb{R}\right),
\end{split}
\right.
\end{equation} 
where $a$ is the age, $\mu>0$ is the mortality of the population and $h\in C\left(\mathbb{R},L^q\left(\left(0,+\infty\right),\mathbb{R}\right)\right)$ (with $\frac{1}{p}+\frac{1}{q}=1$). $W$ is a one-dimensional Brownian motion. Set the Banach space $X:=\mathbb{R}\times L^p\left(\left(0,+\infty\right),\mathbb{R}\right)$ equipped with the usual product norm
\begin{equation*}
	\left\Vert\binom{r}{\varphi}\right\Vert=\left\vert r\right\vert+\left\Vert\varphi\right\Vert_{L^p\left(\left(0,+\infty\right),\mathbb{R}\right)},
\end{equation*}
and set $X_0:=\{0_\mathbb{R}\}\times L^p\left(\left(0,+\infty\right),\mathbb{R}\right)$. Consider a linear operator $\mathcal{A}:D\left(\mathcal{A}\right)\subset X\to X$ defined by
\begin{equation*}
\mathcal{A}
\begin{pmatrix}
0\\
\varphi\\
\end{pmatrix}=
\begin{pmatrix}
-\varphi\left(0\right)\\
-\varphi'-\mu\varphi\\
\end{pmatrix}
\end{equation*}
with
\begin{equation*}
	D\left(\mathcal A\right)=\{0_\mathbb{R}\}\times W^{1,p}\left(\left(0,+\infty\right),\mathbb{R}\right).
\end{equation*}
Notice that $X_0=\overline{D\left(\mathcal{A}\right)}\neq X$. So $\mathcal A$ is non-densely defined on $X$. Also, $\mathcal A$ is Hille-Yosida if and only if $p=1$. By \cite[Lemma 8.1, Lemma 8.3]{PM2009}, $\mathcal{A}$ is a MR operator and we have
\begin{equation*}
 0<\limsup_{\lambda\to+\infty}\lambda^{1/p}\left\|\left(\lambda I-\mathcal A\right)^{-1}\right\|_{\mathcal{L}\left(X\right)}<\infty.
\end{equation*}
We can conclude that $\mathcal{A}_0$, the part of $\mathcal A$ in $X_0$, is Hille-Yosida and generates a $C_0$-semigroup $\{T_{\mathcal{A}_0}\left(t\right)\}_{t\geq0}$. $\mathcal A$ generates an integrated semigroup $\{S_{\mathcal{A}}\left(t\right)\}_{t\geq0}$.

Consider the family of bounded linear operators $\{\mathcal{F}\left(t\right)\}_{t\geq0}\subset\mathcal{L}\left(X_0,X\right)$ given by
\begin{equation*}
	\mathcal{F}\left(t\right)\binom{0}{\varphi}:=\binom{\int_0^{+\infty}h\left(t,a\right)\varphi\left(a\right)da}{0},~\forall\binom{0}{\varphi}\in X.
\end{equation*}
Take
\begin{equation*}
 u^*\left(t\right)=
\begin{pmatrix}
0_{\mathbb{R}^2}\\
u\left(t,\cdot\right)
\end{pmatrix},\quad \omega\left(t\right)=
\begin{pmatrix}
0\\
W\left(t\right)
\end{pmatrix} .
\end{equation*}
Then we can rewrite \eqref{eq5.2} as
\begin{equation}\label{eq57}
	du^*\left(t\right)=\left(\mathcal{A}+\mathcal{F}\left(t\right)\right)u^*\left(t\right)dt+u^*\left(t\right)\circ d\omega\left(t\right),~t\geq t_0,~u^*\left(t_0\right)=\binom{0_{\mathbb{R}^2}}{\varphi}.
\end{equation}
Assume additionally 
\begin{equation*}
\sup_{t\in\mathbb{R}}\Vert h\left(t,\cdot\right)\Vert_{L^q}<+\infty.
\end{equation*}
Then Assumption \ref{as2.1} can be easily checked by H{\"o}lder's inequality.

\begin{definition}
	A function $u^* \in C\left(\mathbb{R}, X_0\right)$ is an integrated solution of \eqref{eq57} if and only if for each $t\geq t_0$,
\begin{equation*}
	\int_{t_0}^t u^*\left(r\right)dr \in D(\mathcal{A}),
\end{equation*}
and
\begin{equation*}
u^*\left(t\right)=u^*\left(t_0\right)+\int_{t_0}^t [\mathcal{A}+\mathcal{F}\left(r\right)]u^*(r) d r+\int_{t_0}^t u^*\left(r\right)\circ d\omega\left(r\right) .
\end{equation*}
We say that $u^*$ is a mild solution of \eqref{eq57} if for each $t\geq t_0$,
\begin{equation*}
u^*\left(t\right)=T_{\mathcal{A}_0}\left(t-t_0\right) u^*\left(t_0\right)+\frac{d}{d t} \int_{t_0}^t S_\mathcal{A}(t-s)\mathcal{F}(s) u^*(s)ds+\frac{d}{d t}\int_{t_0}^t S_\mathcal{A}(t-s)u^*(s)\circ d\omega\left(s\right).
\end{equation*}
\end{definition}
We recall the following definition of an evolution family, which is the generalization of a $C_0$-semigroup to two parameters.
\begin{definition}
	Define $\Delta:=\{\left(t,s\right)\in\mathbb{R}^2:t\geq s\}$. Let $\left(Z,\Vert\cdot\Vert_Z\right)$ be a Banach space. A family of bounded linear operators $\{U\left(t,s\right)\}_{\left(t,s\right)\in\Delta}$ on $Z$ is called an exponentially bounded evolution family if the following conditions are satisfied:
	\begin{itemize}
		\item [$\left(\mathrm{i}\right)$]$U\left(t,t\right)=I_{\mathcal{L}\left(Z\right)}$, $U\left(t,r\right)U\left(r,s\right)=U\left(t,s\right)$ for $t,r,s\in\mathbb{R}$ with $t\geq r\geq s$.
		\item [$\left(\mathrm{ii}\right)$]For each $x\in Z$, the map $\left(t,s\right)\to U\left(t,s\right)x$ is continuous from $\Delta$ into $Z$.
		\item [$\left(\mathrm{iii}\right)$]There exists two constants, $M\geq1$ and $\omega\in\mathbb{R}$, such that $\Vert U\left(t,s\right) \Vert_{\mathcal{L}\left(Z\right)}\leq Me^{\omega\left(t-s\right)}$ for $\forall \left(t,s\right)\in\Delta$.
	\end{itemize}
\end{definition}

Consider the following abstract Cauchy problem with initial time $t_0\in\mathbb{R}$,
\begin{equation}\label{eq38}
	\frac{dv\left(t\right)}{dt}=\mathcal{A}v\left(t\right)+\mathcal{F}\left(t\right)v\left(t\right),~t\geq t_0,~v\left(t_0\right)=\binom{0_{\mathbb{R}^2}}{\varphi}.
\end{equation}
By \cite[Proposition 1.5]{Magal2016}, \eqref{eq38} generates an exponentially bounded unique evolution family $\{U_\mathcal{F}\left(t,s\right)\}_{\left(t,s\right)\in\Delta}\subset X_0$ and $U_\mathcal{F}\left(\cdot,t_0\right)x_0\in C\left([t_0,+\infty),X_0\right)$ is the unique solution of the fixed point problem
\begin{equation*}
	U_\mathcal{F}\left(t,t_0\right)x_0=T_{\mathcal{A}_0}\left(t-t_0\right)+\frac{d}{dt}\int_{t_0}^tS_\mathcal{A}\left(t-s\right)\mathcal{F}\left(s\right)U_\mathcal{F}\left(s,t_0\right)x_0ds, ~t\geq t_0.
\end{equation*}

Define 
\begin{equation*}
	\mathcal C\left(\mathbb{R};X\right):=\left\{f\in C\left(\mathbb R, X\right):\Vert f\Vert_\infty=\sup_{t\in\mathbb R}\Vert f\Vert_X<+\infty\right\}.
\end{equation*}
By \cite[Theorem 5.13]{Magal2016}, we have the following result concerning the exponential dichotomy.
\begin{lemma}
	The following assertions are equivalent
	\begin{itemize}
		\item [$\left(\mathrm{i}\right)$]The evolution family $\{U_\mathcal{F}\left(t,s\right)\}_{\left(t,s\right)\in\Delta}$ has a exponential dichotomy.
		\item [$\left(\mathrm{i}\right)$]For each $f\in\mathcal{C}\left(\mathbb R;X\right)$, there exists a unique integrated solution $v\in\mathcal C\left(\mathbb R;X_0\right)$ of 
		\begin{equation*}
			\frac{dv\left(t\right)}{dt}=\mathcal{A}v\left(t\right)+\mathcal{F}\left(t\right)v\left(t\right)+f\left(t\right),~t\in\mathbb R.
			\end{equation*}
	\end{itemize}
\end{lemma}

With the exponential dichotomy verified, we could study stable and unstable foliations of \eqref{eq5.2} by similar arguments in Theorem \ref{thm4.2}. For the exponential trichotomy, one can refer to Chapter 5 of \cite{PM_2009} about the age-structured model with bifurcation parameter.

\subsection{Stochastic parabolic equation}
We consider the following stochastic parabolic
equation with a nonlinear(and nonlocal) boundary condition
\begin{equation}\label{eq5.1}
	\left\lbrace
\begin{split}
&\frac{\partial u}{\partial t}=\frac{\partial^2u}{\partial x^2}+{{\pi}^2}u\left(t,x\right)+P\left(u\left(t,x\right)\right)+u\left(t,x\right)\circ dW\left(t\right), t\geq0, x>0,\\
&-\frac{u\left(t,0\right)}{\partial x}=Q\left(u\left(t,\cdot\right)\right),
\\
&u\left(0,\cdot\right)=u_0\in L^p\left((0,+\infty),\mathbb{R}\right).
\end{split}
\right.
\end{equation}
where $P: L^p\left((0,+\infty),\mathbb{R}\right)\to L^p\left((0,+\infty),\mathbb{R}\right)$ and $Q: L^p\left((0,+\infty),\mathbb{R}\right)\to \mathbb{R}$ are arbitrary nonlinear maps which are assumed to be globally Lipschitz continuous and $P\left(0\right)=Q\left(0\right)=0$. $W$ is a one-dimensional Brownian motion. Make $X:=\mathbb{R}\times L^p\left((0,+\infty),\mathbb{R}\right)$. Consider  a linear operator $A:D\left(A\right)\subset X\to X$ defined by
\begin{equation*}
A
\begin{pmatrix}
0\\
\varphi\\
\end{pmatrix}=
\begin{pmatrix}
\varphi'\left(0\right)\\
\varphi''\\
\end{pmatrix}
\end{equation*}
with $D\left(A\right)=\{0_\mathbb{R}\}\times W^{2,p}\left(\left(0,+\infty\right),\mathbb{R}\right)$. Notice that $A_0$, the part of $A$ in $\overline{D\left(A\right)}=\{0_\mathbb{R}\}\times L^p\left(\left(0,+\infty\right),\mathbb{R}\right)$, is the generator of the strongly continuous semigroup of bounded operators associated to
\begin{equation*}
	\left\lbrace
\begin{split}
&\frac{\partial u}{\partial t}=\frac{\partial^2u}{\partial x^2},~~ t\geq0, x>0,\\
&-\frac{u\left(t,0\right)}{\partial x}=0,
\\
&u\left(0,\cdot\right)=u_0\in L^p\left((0,+\infty),\mathbb{R}\right).
\end{split}
\right.
\end{equation*}
That is 
\begin{equation*}
	A_0
\begin{pmatrix}
0\\
\varphi\\
\end{pmatrix}=
\begin{pmatrix}
0\\
\varphi''\\
\end{pmatrix}
\end{equation*}
with
\begin{equation*}
	D\left(A_0\right)=\bigg \{\binom{0}{\varphi}\in\{0_\mathbb{R}\}\times W^{2,p}\left(\left(0,+\infty\right),\mathbb{R}\right):\varphi'\left(0\right)=0\bigg\}.
\end{equation*}
Take 
\begin{equation*}
F
\begin{pmatrix}
0\\

\varphi
\end{pmatrix}=
\begin{pmatrix}
Q\left(\varphi\right)\\
P\left(\varphi\right)\\

\end{pmatrix},\quad u^*\left(t\right)=
\begin{pmatrix}
0_{\mathbb{R}^2}\\
u\left(t,\cdot\right)
\end{pmatrix},\quad \omega\left(t\right)=
\begin{pmatrix}
0\\
W\left(t\right)
\end{pmatrix} .
\end{equation*}
Then we can rewrite \eqref{eq5.1} as
\begin{equation*}
\left\lbrace
\begin{split}
&du^*\left(t\right)= \left(\left(A+\pi^2I\right)u^*\left(t\right)+F\left(u^*\left(t\right)\right)\right)dt+u^*\left(t\right)\circ d\omega\left(t\right), ~t\geq0,\\
&u^*\left(0,\cdot\right)=u^*_0\in \overline{D\left(A\right)}:=X_0.
\end{split}
\right.
\end{equation*}
According to \cite[Lemma 6.1]{Magal2016}, $A_0$ is the infinitesimal generator of a $C_0$-semigroup $\{T_{A_0}\left(t\right)\}_{t\geq0}$ on $X_0$. Moreover, $\left(A+\pi^2{I}\right)_0$, the part of $\left(A+\pi^2{I}\right)$, is the infinitesimal generator of a $C_0$-semigroup on $X_0$ denoted by $\{T_{\left(A+\pi^2{I}\right)_0}\left(t\right)\}_{t\geq0}$. By \cite[Proposition 2.5]{PM2007,PM2009}, $\left(A+\pi^2{I}\right)$ generates a integrated semigroup $\{S_{\left(A+\pi^2{I}\right)}\left(t\right)\}_{t\geq0}$. The spectrum of $A_0$ is given by
\begin{equation*}
	\sigma\left(A_0\right)=\left\{-\left(\pi{k}\right)^2:k\in\mathbb{N}\right\}.
\end{equation*}
And by \cite[Lemma 6.3]{Magal2016}, we have the following estimations
\begin{equation*}
0 < \liminf_{\lambda\to+\infty}\lambda^{1/p^*}\left\|\left(\lambda I-A\right)^{-1}\right\|_{\mathcal{L}\left(X\right)}\le
 \limsup_{\lambda\to+\infty}\lambda^{1/p^*}\left\|\left(\lambda I-A\right)^{-1}\right\|_{\mathcal{L}\left(X\right)}<\infty.
\end{equation*}
where $p^*=\frac{2p}{1+p}$. We can obtain that $A$ is non-densely defined and is not a Hille-Yosida operator when $p>1$. Assumption \ref{as2.1} is obviously satisfied by the Lipschitz continuity of $P, Q$. We make another assumption additionally. 
\begin{assumption}\label{as5.1}
	Let $A:D\left(A\right)\subset X\to X$ be a  linear operator on a Banach space $X$. Assume that there exist two constants, $\omega_A\in\mathbb{R}$ and $M_A>0$, such that 
	\begin{itemize}
		\item  [$\left(\mathrm{i}\right)$] $\rho\left(A_0\right)\supset\{\lambda\in\mathbb{C}:\rm Re\left(\lambda\right)>\omega_A\}$ and
	\begin{equation*}
		\left\Vert\left(\lambda-\omega_A\right)\left(\lambda I-A_0\right)^{-1}\right\Vert_{\mathcal{L}\left(X_0\right)}\leq M_A,~\forall\lambda\in\mathbb{C},~ \rm Re\left(\lambda\right)>\omega_A;
	\end{equation*}
		\item [$\left(\mathrm{ii}\right)$]$\left(\omega_A,+\infty\right)\subset \rho\left(A\right)$ and there exists $p^*>1$ such that
	\begin{equation*}
 \limsup_{\lambda\to+\infty}\lambda^{1/p^*}\left\|\left(\lambda I-A\right)^{-1}\right\|_{\mathcal{L}\left(X\right)}<\infty.
\end{equation*}
	\end{itemize}
\end{assumption}
The above asumption can be reformulated by saying that $A_0$ is sectorial and $A$ is $\frac{1}{p^*}$-almost sectorial. By \cite[Lemma 6.1-6.3]{Magal2016}, Assumption \ref{as5.1} is satisfied. Then by \cite[Theorem 3.11]{DP2009}, for fixed $\lambda>\omega_A$, $\hat p>p^*$ and $\tau_0>0$, for each $f\in L^{\hat p}\left([0,\tau_0],X\right)$, the following esitimate holds
\begin{equation*}
	\big\Vert\left(S\diamond f\right)\left(t\right)\big\Vert\leq M_{\beta,\tau_0}\int_0^t\left(t-s\right)^{-\beta}e^{\omega_A\left(t-s\right)}\Vert f\left(s\right) \Vert ~ds,~\forall{t\in[0,\tau_0]},
\end{equation*}	
where $\beta\in\left(1-\frac{1}{p^*},1-\frac{1}{\hat p}\right)$ and $M_{\beta,\tau_0}$ is some positive constant. Therefore
\begin{equation*}
	\big\Vert\left(S\diamond f\right)\left(t\right)\big\Vert\leq \left(M_{\beta,\tau_0}\int_0^t\left(t-s\right)^{-\beta}e^{\omega_A\left(t-s\right)}ds\right)\underset{s\in[0,t]}{\sup}\left\Vert{f\left(s\right)}\right\Vert,~\forall{t\in[0,\tau_0]},
\end{equation*}	
Then by \cite[Lemma 6.4]{Magal2016}, $A$ is a MR operator.
Then we check Assumption \ref{as2.2}. In fact, by \cite[Lemma 6.2]{Magal2016} and \cite[Lemma 2.1]{PM_2009}, we have $\sigma\left(A_0\right)=\sigma\left(A\right)$, then
\begin{equation*}
\begin{split}
	\sigma\left(A+\pi^2I\right)=\sigma\left(A_0+\pi^2I\right)&=\left\{-\left(\pi{k}\right)^2+\pi^2:k\in\mathbb{N}\right\}\\&=\left\{\pi^2, 0, -3\pi^2, -8\pi^2, -15\pi^2,...\right\} ,
\end{split}
\end{equation*}
and each enginvalue $\lambda_k=\left(1-k^2\right)\pi^2$ corresponding to the enginfunction
\begin{equation*}
	\psi_k\left(x\right)=\sin\left(\pi{kx}\right).
\end{equation*}
By Remark \ref{rem}, we could take $\alpha=\pi^2-\varepsilon^*$, $\beta=\pi^2+\varepsilon^*$ and $\gamma=\gamma^*\in\left(0, \pi^2-\varepsilon^*\right)$ for some $\varepsilon^*\in\left(0,\pi^2\right)$. Thus the spectrum of $A_0+\pi^2I$ could be split into three parts $\sigma^{0s}=\left\{\left(1-k\right)^2\pi^2:k=2,3,...\right\},\sigma^{0c}=\left\{0\right\},\sigma^{0u}=\left\{\pi^2\right\}$. So $A_0+\pi^2I$ satisfies the exponential trichotomy condition and the Assumption \ref{as2.2} is satisfied. The center subspace $X_{0c}$, stable subspace $X_{0s}$ and unstable subspace $X_{0u}$ of $X_0$ are $\text{span}\left\{\sin\left(\pi{x}\right)\right\}$, $\text{span}\left\{\sin\left(\pi{nx}\right),n=2,3,...\right\}$, and $\left\{0\right\}$, respectively. 

Therefore, the results obtained in Theorems \ref{thm3.6} and \ref{thm4.2} are applicable and thus justify the existence and smoothness of center invariant  foliations for \eqref{eq5.1}.

\section*{Acknowledgments}
Huang is partially supportted by the National Natural Science Foundation of China (No.12301020), Scientific Research Program Funds of NUDT(No. 22-ZZCX-016). Zeng is partially supported by the National Natural Science Foundation of China (No. 12271177), Natural Science Foundation of Guangdong Province (No. 2023A1515010622). 
 
\section*{References}

\bibliography{mybibfile}

\end{document}